\documentclass[a4paper,11pt]{amsart}
\usepackage{amsmath}
\usepackage{amsthm}
\newtheorem{theorem}{Theorem}[section]
\newtheorem{proposition}[theorem]{Proposition}

\newtheorem{definition}[theorem]{Definition}

\newtheorem{lemma}[theorem]{Lemma}

\theoremstyle{remark}
\newtheorem{rmk}[theorem]{Remark}
\newtheorem{notation}[theorem]{Notation}
\usepackage{amssymb}
\usepackage{enumerate}
\usepackage{graphicx}
\usepackage{tikz}
\usepackage{subfigure}
\usepackage{graphicx}
\usepackage{comment}
\usepackage{pdfpages}
\usepackage{footnote}
\usepackage{makecell}
\usepackage{color}
\usepackage{hyperref}
\makesavenoteenv{tabular}

\usepackage{cancel}

\newcommand{\threed}{{3\mathcal{D}}} %
\newcommand{\threeda}[1]{\threed({#1})}
\newcommand{\threedna}[2]{\threed_{#2}({#1})} %
\newcommand{\MC}[1]{{\mathcal{#1}}}

\newcommand\SternBrocot{
\begin{tikzpicture}[xscale=11,yscale=0.5]
\useasboundingbox (-1/40,4.5) rectangle (1,0);
\scriptsize
\draw (-1/15,4) node[anchor=base] {$\mathcal{SB}_0:$};
\draw (0,4)  -- (1,4);
\draw (0,4.07)  -- (0,3.93);
\draw (1,4.07)  -- (1,3.93);
\draw (0,4.2) node[anchor=base] {$0/1$};
\draw (1,4.2) node[anchor=base] {$1/1$};
\draw[<-,dotted] (0,-0.5) -- (0,4);
\draw[<-,dotted] (1,-0.5) -- (1,4);1

\draw (-1/15,3) node[anchor=base] {$\mathcal{SB}_1:$};
\draw[-] (0,3) -- (1/2,3);
\draw[very thick] (1/2,3) -- (1,3);
\draw (0,3.07)  -- (0,2.93);
\draw (1/2,3.07) -- (1/2,2.93);
\draw (1,3.07)  -- (1,2.93);
\draw (1/2,3.2) node[anchor=base] {$1/2$};
\draw[<-,dotted] (1/2,-0.5) -- (1/2,3);

\draw (-1/15,2) node[anchor=base] {$\mathcal{SB}_2:$};
\draw[-] (0,2) -- (1/3,2);
\draw[very thick] (1/3,2)  --  (2/5,2);
\draw[-] (2/5,2) -- (1/2,2);
\draw[very thick] (1/2,2)  -- (2/5,2);
\draw[-] (2/5,2) -- (2/3,2);
\draw[very thick] (2/3,2)  -- (1,2);
\draw (0,2.07)  -- (0,1.93);
\draw (1/3,2.07) -- (1/3,1.93);
\draw (1/2,2.07) -- (1/2,1.93);
\draw (2/3,2.07) -- (2/3,1.93);
\draw (1,2.07)  -- (1,1.93);
\draw (1/3,2.2)  node[anchor=base] {$1/3$};
\draw (2/3,2.2) node[anchor=base] {$2/3$};
\draw[<-,dotted] (1/3,-0.5) -- (1/3,2);
\draw[<-,dotted] (2/3,-0.5) -- (2/3,2);

\draw (-1/15,1) node[anchor=base] {$\mathcal{SB}_3:$};
\draw[-]     (0,1) --    (1/4,1);
\draw[very thick] (1/4,1)  -- (1/3,1);
\draw[-]     (1/3,1) --  (2/5,1);
\draw[very thick]     (2/5,1) --    (1/2,1);
\draw[-]     (1/2,1) --    (3/5,1);
\draw[very thick] (3/5,1)  -- (2/3,1);
\draw[-]     (2/3,1) --    (3/4,1);
\draw[very thick] (3/4,1)  -- (1,1);
\draw (0,1.07)  -- (0,0.93);
\draw (1/4,1.07)  -- (1/4,0.93);
\draw (1/3,1.07)  -- (1/3,0.93);
\draw (2/5,1.07)  -- (2/5,0.93);
\draw (1/2,1.07)  -- (1/2,0.93);
\draw (3/5,1.07)  -- (3/5,0.93);
\draw (2/3,1.07)  -- (2/3,0.93);
\draw (3/4,1.07)  -- (3/4,0.93);
\draw (1,1.07)  -- (1,0.93);
\draw (1/4,1.2) node[anchor=base] {$1/4$};
\draw (2/5,1.2) node[anchor=base] {$2/5$};
\draw (3/4,1.2) node[anchor=base] {$3/4$};
\draw (3/5,1.2) node[anchor=base] {$3/5$};
\draw[<-,dotted] (1/4,-0.5) -- (1/4,1);
\draw[<-,dotted] (2/5,-0.5) -- (2/5,1);
\draw[<-,dotted] (3/5,-0.5) -- (3/5,1);
\draw[<-,dotted] (3/4,-0.5) -- (3/4,1);

\draw (-1/15,0) node[anchor=base] {$\mathcal{SB}_4:$};
\draw[-]     (0,0) --    (1/5,0);
\draw[very thick] (1/5,0)  -- (1/4,0);
\draw[-]     (1/4,0) --    (2/7,0);
\draw[very thick] (2/7,0)  -- (1/3,0);
\draw[-]     (1/3,0) --    (3/8,0);
\draw[very thick] (3/8,0)  -- (2/5,0);
\draw[-]     (2/5,0) --    (3/7,0);
\draw[very thick]     (3/7,0) --    (1/2,0);
\draw[-]     (1/2,0) --    (4/7,0);
\draw[very thick]     (4/7,0) --    (3/5,0);
\draw[-]     (3/5,0) --    (5/8,0);
\draw[very thick]     (5/8,0) --    (2/3,0);
\draw[-]     (2/3,0) --    (5/7,0);
\draw[very thick]     (5/7,0) --    (3/4,0);
\draw[-]     (3/4,0) --    (4/5,0);
\draw[very thick]     (4/5,0) --    (1,0);
\draw (1/5,0.2) node[anchor=base] {$1/5$};
\draw (2/7,0.2) node[anchor=base] {$2/7$};
\draw (3/8,0.2) node[anchor=base] {$3/8$};
\draw (3/7,0.2) node[anchor=base] {$3/7$};
\draw (4/7,0.2) node[anchor=base] {$4/7$};
\draw (5/8,0.2) node[anchor=base] {$5/8$};
\draw (5/7,0.2) node[anchor=base] {$5/7$};
\draw (4/5,0.2) node[anchor=base] {$4/5$};
\draw[<-,dotted] (1/5,-0.5) -- (1/5,0);
\draw[<-,dotted] (2/5,-0.5) -- (2/5,0);
\draw[<-,dotted] (3/5,-0.5) -- (3/5,0);
\draw[<-,dotted] (4/5,-0.5) -- (4/5,0);
\end{tikzpicture}
}

\newcommand\Farey{
\begin{tikzpicture}[xscale=11,yscale=0.5]
\useasboundingbox (-1/40,4.5) rectangle (1,0);
\scriptsize
\draw (-1/20,4) node[anchor=base] {$\mathcal F_0:$};
\draw (0,4)  -- (1,4);
\draw (0,4.07)  -- (0,3.93);
\draw (1,4.07)  -- (1,3.93);
\draw (0,4.2) node[anchor=base] {$0/1$};
\draw (1,4.2) node[anchor=base] {$1/1$};
\draw[<-,dotted] (0,-0.5) -- (0,4);
\draw[<-,dotted] (1,-0.5) -- (1,4);1

\draw (-1/20,3) node[anchor=base] {$\mathcal F_1:$};
\draw[-] (0,3) -- (1/2,3);
\draw[very thick] (1/2,3) -- (1,3);
\draw (0,3.07)  -- (0,2.93);
\draw (1/2,3.07) -- (1/2,2.93);
\draw (1,3.07)  -- (1,2.93);
\draw (1/2,3.2) node[anchor=base] {$1/2$};
\draw[<-,dotted] (1/2,-0.5) -- (1/2,3);

\draw (-1/20,2) node[anchor=base] {$\mathcal F_2:$};
\draw[-] (0,2) -- (1/3,2);
\draw[very thick] (1/3,2)  --  (1/2,2);
\draw[-] (1/2,2) -- (2/3,2);
\draw[very thick] (2/3,2)  -- (1,2);
\draw (0,2.07)  -- (0,1.93);
\draw (1/3,2.07) -- (1/3,1.93);
\draw (1/2,2.07) -- (1/2,1.93);
\draw (2/3,2.07) -- (2/3,1.93);
\draw (1,2.07)  -- (1,1.93);
\draw (1/3,2.2)  node[anchor=base] {$1/3$};
\draw (2/3,2.2) node[anchor=base] {$2/3$};
\draw[<-,dotted] (1/3,-0.5) -- (1/3,2);
\draw[<-,dotted] (2/3,-0.5) -- (2/3,2);

\draw (-1/20,1) node[anchor=base] {$\mathcal F_3:$};
\draw[-]     (0,1) --    (1/4,1);
\draw[very thick] (1/4,1)  -- (1/3,1);
\draw[-]     (1/3,1) --  (1/2,1);
\draw[very thick]     (1/2,1) --    (2/3,1);
\draw[-]     (2/3,1) --    (3/4,1);
\draw[very thick] (3/4,1)  -- (1,1);
\draw (0,1.07)  -- (0,0.93);
\draw (1/4,1.07)  -- (1/4,0.93);
\draw (1/3,1.07)  -- (1/3,0.93);
\draw (1/2,1.07)  -- (1/2,0.93);
\draw (2/3,1.07)  -- (2/3,0.93);
\draw (3/4,1.07)  -- (3/4,0.93);
\draw (1,1.07)  -- (1,0.93);
\draw (1/4,1.2) node[anchor=base] {$1/4$};
\draw (3/4,1.2) node[anchor=base] {$3/4$};
\draw[<-,dotted] (1/4,-0.5) -- (1/4,1);
\draw[<-,dotted] (3/4,-0.5) -- (3/4,1);

\draw (-1/20,0) node[anchor=base] {$\mathcal F_4:$};
\draw[-]     (0,0) --    (1/5,0);
\draw[very thick] (1/5,0)  -- (1/4,0);
\draw[-]     (1/4,0) --    (1/3,0);
\draw[very thick] (1/3,0)  -- (2/5,0);
\draw[-]     (2/5,0) --    (1/2,0);
\draw[very thick] (1/2,0)  -- (3/5,0);
\draw[-]     (3/5,0) --    (2/3,0);
\draw[very thick]     (2/3,0) --    (3/4,0);
\draw[-]     (3/4,0) --    (4/5,0);
\draw[very thick]     (4/5,0) --    (1,0);
\draw (1/5,0.2) node[anchor=base] {$1/5$};
\draw (2/5,0.2) node[anchor=base] {$2/5$};
\draw (3/5,0.2) node[anchor=base] {$3/5$};
\draw (4/5,0.2) node[anchor=base] {$4/5$};
\draw[<-,dotted] (1/5,-0.5) -- (1/5,0);
\draw[<-,dotted] (2/5,-0.5) -- (2/5,0);
\draw[<-,dotted] (3/5,-0.5) -- (3/5,0);
\draw[<-,dotted] (4/5,-0.5) -- (4/5,0);

\end{tikzpicture}
}

\newcommand\palindomic{
\begin{tikzpicture}[xscale=11,yscale=0.5]
\useasboundingbox (-1/40,4.5) rectangle (1,0);
\scriptsize
\draw (-1/20,4) node[anchor=base] {$F_0:$};
\draw (0,4)  -- node[above,midway] {$\varepsilon$} (1,4) ;
\draw (0,4.07)  -- (0,3.93);
\draw (1,4.07)  -- (1,3.93);
\draw (0,4.2) node[anchor=base] {$0/1$};
\draw (1,4.2) node[anchor=base] {$1/1$};
\draw[<-,dotted] (0,-0.) -- (0,4);
\draw[<-,dotted] (1,-0.) -- (1,4);

\draw (-1/20,3) node[anchor=base] {$F_1:$};
\draw[-] (0,3) -- node[above,midway] {$0$}  (1/2,3);
\draw[very thick] (1/2,3) -- node[above,midway] {$1$}  (1,3);
\draw (0,3.07)  -- (0,2.93);
\draw (1/2,3.07) -- (1/2,2.93);
\draw (1,3.07)  -- (1,2.93);
\draw (1/2,3.3) node[anchor=base] {$\tfrac{1}{2}$};
\draw[<-,dotted] (1/2,-0.) -- (1/2,3);

\draw (-1/20,2) node[anchor=base] {$F_2:$};
\draw[-] (0,2) -- node[above,midway] {$00$}  (1/3,2);
\draw[very thick] (1/3,2)  --  node[above,midway] {$01$}  (1/2,2);
\draw[-] (1/2,2) -- node[above,midway] {$10$} (2/3,2);
\draw[very thick] (2/3,2)  -- node[above,midway] {$11$}  (1,2);
\draw (0,2.07)  -- (0,1.93);
\draw (1/3,2.07) -- (1/3,1.93);
\draw (1/2,2.07) -- (1/2,1.93);
\draw (2/3,2.07) -- (2/3,1.93);
\draw (1,2.07)  -- (1,1.93);
\draw (1/3,2.3)  node[anchor=base] {$\tfrac{1}{3}$};
\draw (2/3,2.3) node[anchor=base] {$\tfrac{2}{3}$};
\draw[<-,dotted] (1/3,-0.) -- (1/3,2);
\draw[<-,dotted] (2/3,-0.) -- (2/3,2);

\draw (-1/20,1) node[anchor=base] {$F_3:$};
\draw[-]     (0,1) --  node[above,midway] {$000$}   (1/4,1);
\draw[very thick] (1/4,1)  -- node[above,midway] {$001$}  (1/3,1);
\draw[-]     (1/3,1) -- node[above,midway] {$010$}  (1/2,1);
\draw[very thick]     (1/2,1) -- node[above,midway] {$101$}     (2/3,1);
\draw[-]     (2/3,1) -- node[above,midway] {$110$}    (3/4,1);
\draw[very thick] (3/4,1)  -- node[above,midway] {$111$}  (1,1);
\draw (0,1.07)  -- (0,0.93);
\draw (1/4,1.07)  -- (1/4,0.93);
\draw (1/3,1.07)  -- (1/3,0.93);
\draw (1/2,1.07)  -- (1/2,0.93);
\draw (2/3,1.07)  -- (2/3,0.93);
\draw (3/4,1.07)  -- (3/4,0.93);
\draw (1,1.07)  -- (1,0.93);
\draw (1/4,1.3) node[anchor=base] {$\tfrac{1}{4}$};
\draw (3/4,1.3) node[anchor=base] {$\tfrac{3}{4}$};
\draw[<-,dotted] (1/4,-0.) -- (1/4,1);
\draw[<-,dotted] (3/4,-0.) -- (3/4,1);

\end{tikzpicture}
}

\DeclareMathOperator{\supp}

\newcounter{PropA}
\renewcommand{\thePropA}{\Alph{PropA}}

\usepackage{multirow}

\usepackage{thmtools}
\usepackage{thm-restate}

\begin{document}
\title[Lochs-type theorems beyond positive entropy]{Lochs-type theorems beyond positive entropy}

\author[V. Berth\'e]{Val\'erie Berth\'e}

\address{Universit\'e de Paris, CNRS, IRIF, F-75006 Paris, France}
\email{berthe@irif.fr}

\author[E. Cesaratto]{Eda Cesaratto}
\address{CONICET and Instituto del Desarrollo Humano, Universidad Nacional de General Sarmiento, Buenos Aires, Argentina}
\email{ecesaratto@campus.ungs.edu.ar}

\author[P. Rotondo]{Pablo Rotondo}
\address{Laboratoire d'Informatique Gaspard-Monge, Universit\'e Gustave Eiffel, Champs-sur-Marne, France} %
\email{pablo.rotondo@univ-eiffel.fr}

\author[M. D. Safe]{Mart\'in D.\ Safe}
\address{Departamento de Matem\'atica, Universidad Nacional del Sur (UNS), Bah\'ia Blanca, Argentina, and INMABB, Universidad Nacional del Sur (UNS)-CONICET, Bah\'ia Blanca, Argentina}
\email{msafe@uns.edu.ar}

\date{\today}
\thanks{This work was supported by the Agence Nationale de la Recherche through the project CODYS (ANR-18-CE40-0007) and STIC AMSUD project 20STIC-06. The authors are members of LIA SINFIN (ex-INFINIS), 
Universit\'e de Paris-CNRS/Universidad de Buenos Aires-CONICET. E.~Cesaratto was partially supported by Grant UNGS 30/3307. M.D.~Safe was partially supported by Grant UNS PGI 24/L115.}

\begin{abstract}
Lochs' theorem and its generalizations are conversion theorems that relate   %
the number of digits determined in one expansion of a real number as a function of the number of digits given in some other expansion. In its  original version, Lochs' theorem related decimal expansions  with  continued fraction expansions. Such conversion results can also be stated for sequences of interval partitions under suitable assumptions, with results holding almost everywhere, or in measure, involving the entropy. This  is the viewpoint we develop here. In order to deal with sequences of partitions beyond positive entropy, this paper introduces the notion of log-balanced sequences of partitions, together with their weight functions. These are sequences of interval partitions such that the logarithms of the measures of their intervals at each depth are roughly the same. We then state Lochs-type theorems which work even in the case of zero entropy,  in particular for several important log-balanced sequences of partitions of a number-theoretic  nature.
\end{abstract}

\maketitle


\smallskip
\section{Introduction} 

Lochs' theorem \cite{Lochs:63}   is  a probabilistic  statement     about  base changes. This
conversion theorem relates, almost everywhere, the  relative speed
of approximation of decimal and regular continued  fractions expansions to the quotient of the entropies of their respective dynamical
systems.  More generally,  Lochs-type theorems  amount to  compare  the number of digits determined in one expansion  of a real number  as a function of the number of digits given in some other expansion. %
 
Lochs' theorem and its extensions have  given rise to  a rich literature. The original setting of Lochs' theorem has been considered in \cite{Faivre:97}   where  an error term is provided based on the  use of a Perron--Frobenius type  operator; see also \cite{Faivre:98} for  a central limit theorem, \cite{Faivre:01,Wu:06} for the  case of   numbers $x$  having a \emph{L\'evy's constant}
$\lim_{n\to\infty}(\log q_n(x))/n$ (including the case of  quadratic numbers), and    \cite{Wu:08} for an iterated   logarithm  law.  Analogous results for the   case  of the   beta-numeration and continued fractions have been established in  \cite{BI:08bis,LiWu:08,FWL:16,FWL:19}  and in
 \cite{LiWu:08bis} in the case of  the beta-numeration and of the continued fractions for formal power series with coefficients in a finite field.
 More general transformations have then  been  considered in \cite{DaField:01,Dajani,Dajanibis}
 for number theoretic fibred systems. Most of these results are stated in the case where both dynamical systems to be compared have positive entropy.

In the present paper we follow    the  viewpoint developed in \cite{DaField:01,Dajani}
 where   Lochs' theorem is understood as a way to see
 how decimal intervals fit into the fundamental intervals provided
 by the continued fraction expansion. Here also, and as stressed in \cite{DaField:01},  proofs are based on measure-theoretic covering arguments and not on the dynamics of specific maps.  We introduce the notion of log-balanced sequences of partitions of the unit interval, inspired by 
 \cite{DaField:01} (where the terminology `equipartition' is used). Roughly speaking, that a sequence $\mathcal P=\{P_n\}_{n\in\mathbb N}$ of partitions of the unit interval is log-balanced means that the intervals of each partition $P_n$ have approximately the same measure  $e^{-f(n)}$ as $n\to\infty$ (a.e.\ or in measure), where $f$ is called a weight function (see Definition~\ref{def:balanced})
We  then  deal with  the Lochs index, a  now classical parameter in base changing,  denoted here  by $L_n$ and defined  for two sequences of partitions as follows. Let $\mathcal P^1=\{P^1_n\}_{n\in\mathbb N}$ and $\mathcal P^2=\{P^2_n\}_{n\in\mathbb N}$ be two sequences of partitions of the unit interval. For each $i=1,2$, let $E^i$ be the set of endpoints of all the intervals of the partitions in $\mathcal P^i$ and, for each $x\in[0,1]\setminus E^i$, let $I^i_n(x)$ be the interval of $P_n^i$ containing $x$. Then, the Lochs index is defined by
 \begin{equation} L_n(x,\mathcal{P}^{1} ,\mathcal{P}^{2} ):=\sup\{\ell\in\mathbb N\mid I^1_n(x)\subseteq I^2_\ell(x)\}, \end{equation}
for each $x\in[0,1]\setminus(E^1\cup E^2)$.(Strictly speaking, we work with \emph{topological} partitions, see Section~\ref{subsec:defbalanced}.)

The behaviour (a.e.\ or in measure) of Lochs' index is usually described in the literature in terms of entropy (we discuss the notion of entropy in Section \ref{subsec:defbalanced} and \ref{subsec:fp}).
Our work is inspired by the following result, which is Theorem 4 from \cite{DaField:01}. 
Let $\lambda$ be a Borel probability measure on $[0,1]$. If $\mathcal P^1$ and $\mathcal P^2$ are sequences of partitions having positive a.e.\ entropies $h_1$ and $h_2$ with respect to $\lambda$, then
 \[\lim_{n\to \infty}\frac{L_n(x,\mathcal{P}^1,\mathcal{P}^2)}{n}=\frac{h_1}{h_2}\quad a.e.\ (\lambda) .\]
As Dajani and Fieldsteel~\cite{DaField:01}, we also consider a variant in measure of the above result.

Our main results, Theorems \ref{thm:resultae} and \ref{thm:resultinmeasure}, work as follows. Let $\mathcal P^1$ and $\mathcal P^2$ be two log-balanced sequences of partitions with weight functions $f_1$ and $f_2$, both of which either a.e.\ or in measure. We give sufficient conditions in order to ensure that
\begin{equation}\label{eq:Lochs-type}
  \lim_{n\to\infty}\frac{f_2(L_n(x,\mathcal P^1,\mathcal P^2))}{f_1(n)}=1
\end{equation}
either a.e.\ or in measure ($\lambda$), respectively. %

Our results extend the results of  \cite{DaField:01}, by  providing  asymptotic relations as \eqref{eq:Lochs-type} for  partitions which may not have positive entropy. The hypotheses of our main results regard separately 
the weight functions $f_1$ and $f_2$. 

For an illustration,  consider the sequence of partitions given by the binary numeration, i.e.,  $\mathcal P^1=\mathcal B= \{B_n\}_{n \in {\mathbb N}}$ 
with \[ B_n=\left\{ \left(0,\frac 1{2^n}\right),\left(\frac 1{2^n},\frac 2{2^n}\right), \ldots, \left(\frac{2^n-1}{2^n},1\right)\right\},\]
and the Farey sequence of partitions $\mathcal P^2=\mathcal F=\{F_n\}_{n \in \mathbb N}$ (the latter of which has zero entropy, see Section \ref{subsec:nfz}), where each $F_n$ is determined by the set $E_n^{\mathcal F}$ of endpoints of its intervals, given as follows:
\[ E^{\mathcal F}_n:= \left\{ \frac p q:\,p,q \geq 1,\  \gcd(p,q)=1\text{ and }q \leq n+1\right\}. \]

Our results imply that, with respect to the Lebesgue measure,
\begin{align*}
  \lim_{n\to\infty}\frac{2\log L_n(x,\mathcal B,\mathcal F)}{(\log 2)n}=1\quad\text{ a.e.}
\end{align*}
(where by $\log$ we mean the natural logarithm).
Note that a change of scale is performed on the Loch's index 
$L_n(x,\mathcal B,\mathcal F)$ which occurs   as 
$\log L_n(x,\mathcal B,\mathcal F)$ compared to the index $n$.

Our main examples of sequences of partitions (see Section \ref{sec:balanced}) are of an arithmetic nature. They are  associated with   numeration systems (see Section \ref{subsec:num}),
  dynamically, with fibred systems (see Section \ref{subsec:fibred}), and more specifically with   continued fractions with  the Stern-Brocot tree 
 (see Section~\ref{subsec:fz}), the
 Farey sequence (see Section~\ref{subsec:nfz}), and a partition related to the three distance theorem (see Section~\ref{subsec:3D}). 

The study of base changes and of the Lochs index opens a  large  scope of potential  applications. Indeed, sequences of partitions  can model  numeration systems (see Section \ref{subsec:num}),  as well as sources for the production of digits;     as  an example, consider    the  dichotomic selection on words for selection algorithms developed in  \cite{ADLV:19} and see  also Section \ref{sec:conclusion}. A further  motivation is  the dynamic generation of characteristic Sturmian words of uniform random parameters (see  %
Section~\ref{sec:conclusion} for more  on this  topic).

\bigskip
\paragraph*{\bf Plan of the article}
The key notion of  a  log-balanced sequence of  partitions is  introduced in Section \ref{subsec:defbalanced} together with   the weight  functions  for our main examples of sequences of partitions. The definition of the Lochs index is given   in Section \ref{subsec:lochs} which also presents  the main theorems of this paper.
Section~\ref{sec:bal-Lochs} gives some basic results on log-balanced sequences of partitions. The examples of sequences of partitions  considered in the present paper are detailed in Section \ref{sec:balanced}, with in particular  
 the notion of  a numeration system   given by  a  sequence of     partitions  in Section~\ref{subsec:num}
   and a discussion on 
 fibred systems in Section   \ref{subsec:fibred}.
 For the case of conversions between the Farey and the continued fractions sequences of partitions, explicit formulas for the corresponding Lochs indexes are given in Section~\ref{sec:FCF}. As a consequence, we derive  in particular a Lochs-type theorem for the conversion from the continued fraction to the Farey sequences of partitions in law (see Section~\ref{sssec:CF->F}).
The proofs of our main general conversion results  are  given in  Section  \ref{sec:convtheo}.
We conclude   this paper with open questions and by developing   connections with  sources and tries in Section \ref{sec:conclusion}.

\bigskip

\renewcommand\thesection{\arabic{section}}
\setcounter{section}{1}
\section{Log-balancedness and Lochs-type theorems}

The aim of this section is to introduce the main notions and state our main results. In Section~\ref{subsec:defbalanced}, we define the  notions of log-balancedness  and weight function. As far as we know, these notions are  introduced in this paper for the first time in the context of Lochs' index. We also relate  weight functions  to the   entropies of sequences of partitions when they are positive. Theorem~\ref{thm:mainweights} presents our main examples of log-balanced sequences of partitions  together with  their weight functions. In Section~\ref{subsec:lochs}, we formally define the notion of Lochs' index,  following the classical approach of the literature. Finally, in the same section, we state our main results.

\subsection{Log-balanced sequences of partitions}\label{subsec:defbalanced} 
We first fix some general notation and definitions. We consider that $ 0 \in {\mathbb N}$. The notation  $\overline{A}$  stands for  the usual topological closure of $A$. The Lebesgue measure is denoted by $\vert\cdot\vert$.

A \emph{topological partition of $[0,1]$} is a collection of pairwise disjoint open nonempty intervals so that the union of their closures equals $[0,1]$. Notice that such a collection is necessarily at most countable. %

By a \emph{sequence of partitions}  $\mathcal{P}=\{P_n\}_{n\in \mathbb N}$, we mean a sequence of topological partitions of $[0,1]$. We refer to the intervals in $P_n$ as the {\em %
intervals at depth $n$} of $\mathcal P$. We say that  $\mathcal P$ is \emph{self-refining} if, for each $n\in\mathbb N$, each interval of $P_{n+1}$ is contained in (and possibly equal to) an interval  of $P_{n}$. A sequence of partitions $\mathcal P=\{P_n\}_{n\in\mathbb N}$ is \emph{strictly self-refining} if it is self-refining and $P_{n+1}\neq P_n$ for each $n\in\mathbb N$; or, equivalently, for each $n\in \mathbb N$, there exists an interval $I_n$ in $P_n$ which is the union of at least two different intervals in $P_{n+1}$. We denote by $E^\mathcal P$ the set of endpoints of the intervals in the partitions in $\mathcal P$.

Given a sequence of partitions $\mathcal P$, $n\in\mathbb N$, and $x\in[0,1]\setminus E^{\mathcal P}$, we denote by $I_n^{\mathcal P}(x)$ the unique %
interval in $P_n$ to which $x$ belongs. If there is no risk of ambiguity, we denote $I_n^{\mathcal P}(x)$ simply by $I_n(x)$ and $E^{\mathcal P}$ simply by $E$.

\begin{definition}\label{def:balanced}
Let $\lambda$ be a  Borel probability measure on $[0,1]$.
Let $\mathcal P$ be a sequence of partitions. We say that $\mathcal P$ is \emph{log-balanced a.e.}\ (resp.\ \emph{in measure}) with respect to $\lambda$ if $\lambda(E)=0$ and there is some function $f:\mathbb N\to\mathbb R$ such that $f(n)\to+\infty$ as $n\to\infty$ and
\[ \lim_{n\to\infty}\frac{-\log\lambda(I_n(x))}{f(n)}=1\quad a.e.\ (\text{resp.\ in measure})\quad (\lambda). \]
If so, $f$ is called a \emph{weight function of $\mathcal P$ a.e.} (resp.\ \emph{in measure}) with respect to $\lambda$.
\end{definition}

The weight function is not unique. If $f$ is a weight function for a sequence of partitions $\mathcal P$ with respect to some measure $\lambda$, any other function $g:\mathbb N\to \mathbb R$ so that $f(n)/g(n)\to 1$ as $n\to \infty$ is also a weight function with respect to $\lambda$. In the particular case of an a.e.\ log-balanced self-refining system of partitions, the weight function might be chosen to be nondecreasing. This is proved in Proposition~\ref{prop:nonmelangeepart1}.

When, in the above definition, the weight $f$ is linear, one recovers the following notions of entropy, as given in \cite{DaField:01}, which are reminiscent of the Shannon--McMillan--Breiman theorem (see Theorem \ref{teo:ShannonMcMillanBreiman}).

\begin{definition}
Let $\mathcal P=\{P_n\}_{n \in \mathbb N}$ be a sequence of partitions and $\lambda$ be a Borel probability measure $\lambda$ on the unit interval. We say  $\mathcal P$ \emph{has entropy $h$ a.e.\ (resp.\ in measure) with respect to $\lambda$} if $\lambda(E)=0$ and
\begin{equation}\label{eq:def-entropy}
\lim_{n\to\infty}\frac{-\log\lambda(I_n(x))}{n}= h\quad\text{a.e.\ (resp.\ in measure)\ }(\lambda).
\end{equation}
\end{definition}

\begin{rmk}
The assumption $\lambda(E)=0$ is not given in ~\cite[Definition 2]{DaField:01}; nevertheless, it is implicit in it. The reason is that in that work, the partitions are made up of semi-open intervals in such a way that for any $x\in[0,1)$ and each $n\in\mathbb N$ there is a well defined $I_n(x)$ in $P_n$ such that $x\in I_n(x)$. Thus, in that  context, \eqref{eq:def-entropy} implies $\lambda(\{x\})=0$ for every $x\in[0,1)$ (cf.\ Proposition~\ref{prop:atom} of this work). Notice also that in ~\cite{DaField:01}, $\lambda([0,1))=1$.\end{rmk}

\subsubsection*{Main examples of weight functions.} Our main examples of sequences of partitions are detailed in Section~\ref{sec:balanced}. It includes  number-theoretic fibred systems, especially the sequence of partitions associated with continued fractions and the binary numeration system.  We also study the Farey sequence of partitions given by the classical Farey sequence. We are interested in it because of its strong relation with Sturmian words (see e.g. \cite{VB:96}). In addition, we study the Stern-Brocot sequence of partitions associated with the Stern-Brocot tree. Finally, we also study what we call three-distance sequences of partitions parameterized by an irrational $\alpha\in(0,1)$, associated with the family of Kronecker-Weyl sequences $n\mapsto\{n\alpha\}$. The precise definition of all of these sequences of partitions are given throughout Section~\ref{sec:balanced}. Also throughout that section, the log-balancedeness of each of these sequences is analyzed, thus obtaining the results summarized in the following theorem.

\begin{theorem} \label{thm:mainweights}
The following assertions hold with respect to the Lebesgue measure:
\begin{enumerate}[(i)]
    \item The binary partition $\mathcal B$ is a.e.\ log-balanced with weight function
    \[ f_{\mathcal B}(n)=(\log 2)n. \]
    
    \item The continued fractions partition $\mathcal{CF}$ is a.e.\ log-balanced with weight function
    \[ f_{\mathcal{CF}}(n)=\frac{\pi^2}{6\log 2} n. \]
    
    \item More generally, if $\mathcal H$ is any sequence of partitions having positive entropy $h$ a.e. (resp.\ in measure), then $\mathcal H$ is log-balanced a.e.\ (resp.\ in measure) with weight function
    \[  f_{\mathcal H}(n)=hn. \]
    
    \item The Farey partition $\mathcal F$ is a.e.\ log-balanced with weight function
    \[ f_{\mathcal F}(n)=2\log n\quad\text{for each }n\geq 1. \]
    
    \item The three-distance partition $3\mathcal D(\alpha)$ is a.e.\ log-balanced, for a.e.\ $\alpha\in(0,1)$, with weight function
    \[ f_{\threeda{\alpha}}(n)=\log n\quad\text{for each }n\geq 1. \]
    However, there is an uncountable set of numbers $\alpha$ for which  $\threeda{\alpha}$ is not even log-balanced in measure. 
    
    \item The Stern-Brocot partition $\mathcal{SB}$ is log-balanced in measure with weight function
    \[ f_{\mathcal{SB}}(n)=\frac{\pi^2}6\frac{n}{\log n}\quad\text{for each }n\geq 2. \]
    Nevertheless, the $\mathcal{SB}$ partition is not log-balanced a.e.
\end{enumerate}
\end{theorem}

The fact  that we work  with the logarithm of the measures of the intervals (and not with  the measures themselves)  allows for some non-uniformity for the measures
of the intervals. This is illustrated in particular by the case of the Stern-Brocot sequence of partitions, %
which is log-balanced in measure (but not a.e.) with respect to the Lebesgue measure. %

\subsection{Lochs-type theorems}\label{subsec:lochs}

We start by giving a precise definition of Lochs' index, following the approach of \cite{Dajani} and \cite{DaField:01}. If $\mathcal P^1$ and $\mathcal P^2$ are sequences of partitions, then, for each $i=1,2$, we denote by $E^i$ the set of endpoints of all the intervals of the partitions in $\mathcal P^i$ and, for each $x\in[0,1]\setminus E^i$, $I^i_n(x)$ the interval of $P_n^i$ containing $x$.

 \begin{definition} \label{def:Lochs}
 If $\mathcal P^1$ and $\mathcal P^2$ are sequences of partitions and $n\in\mathbb N$, the {\em Lochs index} is defined as
 \[ L_n(x,\mathcal{P}^{1} ,\mathcal{P}^{2} ):=\sup\{\ell\in\mathbb N\mid I^1_n(x)\subseteq I^2_\ell(x)\},\quad\text{for each }x\in[0,1]\setminus(E^1\cup E^2). \]
 \end{definition} 
When both sequences of partitions involved are log-balanced a.e.\ or in measure with respect to the same Borel probability measure $\lambda$, the Lochs index is a.e. finite with respect to $\lambda$. This will be  proved in Proposition \ref{prop:Lochsisfinite}. 

\subsubsection{Results for general sequences of partitions}
We first recall  results from the literature. 
Let us consider the sequence of partitions $\mathcal P^1=\mathcal D=\{D_n\}$ of the decimal numeration system, i.e., 
\[ D_n=\left\{ \left(0,\frac 1{10^n}\right),\left(\frac 1{10^n},\frac 2{10^n}\right), \ldots, \left(\frac{10^n-1}{10^n},1\right)\right\},\] and let $\mathcal P^2$ be the classical continued fraction sequence of partitions $\mathcal{CF}$ (as given in Section~\ref{subsubsec:cf}). The following is the classical Lochs' theorem.

\begin{theorem}[{\cite{Lochs:63}}] The following holds with respect to the Lebesgue measure:
\[ \lim_{n\to\infty}\frac{L_n(x,\mathcal{CF},\mathcal D)}{n}=\frac{6\log 2\log 10}{\pi^2}\quad\text{a.e.} \]
\end{theorem}

Dajani and Fieldsteel~\cite{DaField:01} proved theorems like the above one but in a more general setting, involving two sequences of partitions $\mathcal P_1$ and $\mathcal P_2$, both of positive entropies a.e.\ or in measure. %
Their results are as follows.

\begin{theorem}[{\cite[Theorem 4]{DaField:01}}] If $\mathcal P^1$ and $\mathcal P^2$ are sequences of partitions having positive a.e.\ entropies $h_1$ and $h_2$ with respect to some Borel probability measure $\lambda$ on $[0,1]$, then
 \[\lim_{n\to \infty}\frac{L_n(x,\mathcal{P}^1,\mathcal{P}^2)}{n}=\frac{h_1}{h_2}\quad a.e.\ (\lambda) .\]
\end{theorem}

\begin{theorem}[{\cite[Theorem 6]{DaField:01}}] If $\mathcal P^1$ and $\mathcal P^2$ are sequences of partitions having positive entropies $h_1$ and $h_2$ in measure with respect to some Borel probability measure $\lambda$ on $[0,1]$ and $\mathcal P^2$ is self-refining, then
 \[\lim_{n\to \infty}\frac{L_n(x,\mathcal{P}^1,\mathcal{P}^2)}{n}=\frac{h_1}{h_2}\quad\text{in measure}\ (\lambda) .\]
\end{theorem}

We now state our main results which are extensions of the above theorems to log-balanced sequences of partitions.

\begin{restatable*}{theorem}{mainResultae}\label{thm:resultae}
Let $\mathcal P^1$ and $\mathcal P^2$ be two a.e.\ log-balanced sequences of partitions with weight functions $f_1$ and $f_2$, respectively, with respect to some Borel probability measure $\lambda$ on $[0,1]$ such that all the following assertions hold:
\begin{enumerate}[(i)]

    \item\label{it:delta} $\sum_{n=1}^\infty e^{-\delta f_1(n)} < \infty$ for every $\delta>0$;

    \item $f_2$ is nondecreasing;
    
    \item\label{it:f2equivalentn1} $f_2(n+1)-f_2(n)=o(f_2(n))$ as $n\to \infty$.  
\end{enumerate}
Then,
\[
  \lim_{n\to\infty} \frac{f_2(L_n(x,\mathcal P^1,\mathcal P^2))}{f_1(n)}=1\quad\text{a.e.\ }(\lambda).
\]
\end{restatable*}

\begin{rmk} \label{rmk:ae} The assumptions of this  theorem   can be easily checked in natural instances.   

\begin{enumerate}
    \item Assumption  (\ref{it:delta}) above holds for any $f$ satisfying that $\lim_{n\to \infty} f(n)/\log n=\infty$.
In contrast, it does not hold if $\limsup_{n\to \infty} f(n)/\log n$ is finite.
\item If $\mathcal{P}_2$ is self-refining, then the weight function $f_2$ can be chosen to be nondecreasing (see Proposition~\ref{prop:nonmelangeepart1}).

\item A weight function $f_2$ satisfies assumption (\ref{it:f2equivalentn1}) above if and only if 
   $f_2(n+1)/f_2(n)\to 1$ as $n\to\infty$ or, equivalently,  $\sqrt[1/n]{|f_2(n)|}\to 1$ as $n\to \infty$.
\end{enumerate}

\end{rmk}

\begin{restatable*}{theorem}{mainResultim}\label{thm:resultinmeasure}
Let $\mathcal P^1$ and $\mathcal P^2$ be two sequences of partitions that are log-balanced in measure, having  respective  weight functions $f_1$ and $f_2$, with respect to some Borel probability measure $\lambda$ on $[0,1]$ such that all the following assertions hold:

\begin{enumerate}[(i)]
    \item $\mathcal P^2$ is self-refining;
    
    \item $f_2$ is nondecreasing;
    
    \item\label{it:resultainm3} $f_2(n+1)-f_2(n)=o(f_2(n))$ as $n\to\infty$.
\end{enumerate}
Then,
\[
  \lim_{n\to\infty} \frac{f_2(L_n(x;\mathcal P^1,\mathcal P^2))}{f_1(n)}=1\quad\text{in measure\ }(\lambda).
\]
\end{restatable*}

\begin{rmk}
In this case, there are no assumptions on $f_1$ apart from it being a weight function of $\mathcal P^1$ with respect to $\lambda$. Possible assumptions equivalent to assertion (\ref{it:resultainm3}) are given in Remark \ref{rmk:ae} above.
\end{rmk}

 The above theorems are proved in Section \ref{sec:convtheo}. To prove them, we deal with the corresponding superior and inferior limits separately. This gives four propositions. The conditions in their statements are weaker than those asked in Theorems~\ref{thm:resultae} and \ref{thm:resultinmeasure} but, for the propositions regarding the inferior limits,  are also more technical. The first two, Propositions  \ref{prop:limsupae} and \ref{prop:liminfae}, consider a.e.\ convergence and 
 prove Theorem~\ref{thm:resultae}.  Propositions \ref{prop:limsupiinm} and  \ref{prop:liminm} prove the result in measure, namely Theorem~\ref{thm:resultinmeasure}.

\subsubsection{Positive entropy versus Farey and Stern-Brocot}\label{ssec:F&SB}
We focus here on  conversion  results   involving  the Farey and the Stern-Brocot  sequences of partitions.

In the case $\mathcal P^2$ is the Farey sequence of partitions $\mathcal{F}$ (see Section~\ref{subsec:nfz}) and $\lambda$ is the Lebesgue measure, Theorem~\ref{thm:resultae} applies when $\mathcal P^1$ is any sequence of partitions $\mathcal H$ of positive entropy $h$ a.e.\ and yields the following:   
\begin{equation}\label{eq:H->F} \lim_{n\to\infty}\frac{\log L_n(x,\mathcal H,\mathcal F)}{n}=\frac h2\quad\text{a.e.},
\end{equation}
with respect to the Lebesgue measure. In particular, if $\mathcal H$ is the binary sequence of  partitions  $\mathcal B$ or the continued fraction sequence $\mathcal{CF}$, then $h$ is $\log 2$ or $\pi^2/(6\log 2)$, respectively. Note that this result  admits a  nice interpretation in terms of Sturmian words: because of the connection between Sturmian words and the Farey sequences (see Section~\ref{subsec:sturm}), the assertion for $\mathcal H=\mathcal B$ means that the first $n$ digits in the binary expansion of some $x\in[0,1]$ gives the prefix of length roughly $2^{n/2}$ of the characteristic Sturmian word associated with $x$. Concerning  the case  ${\mathcal H}=\mathcal{CF}$,  note also that using classical results about the statistics of continued fractions expansions, a normal law is even  proved in Theorem~\ref{teo:prob-term-cf-farey} for the Lochs index $L_n(x,\mathcal{CF},\mathcal{F})$.

Notice that assumption~\eqref{it:delta}  of Theorem~\ref{thm:resultae} is not satisfied in the case where $\mathcal P^1=\mathcal F$,   because the corresponding weight function is $f_{\mathcal{F}}(n)=2\log n$ (see Theorem ~\ref{thm:mainweights} and  Remark~ \ref{rmk:ae}). 
In Section~\ref{prop:fare-to-cf}, we show, by alternative means, that
\[
\lim_{n\to\infty}\frac{L_n(x,\mathcal{F},\mathcal{CF})}{\log n} = \frac {12\log 2}{\pi^2}\quad\text{a.e.,}
\]
with respect to the Lebesgue measure. In this case, our general results only provide convergence in measure.  The fact that we may obtain a stronger convergence is not surprising in such a case: the two partitions, continued fractions and Farey, are strongly related in their constructions (see e.g., Lemma~\ref{lem:ICFIF}),  while the proofs of the general results are based on the comparison of the growth of the measures of the  intervals.
 
Lastly, the case of the Stern-Brocot sequence of partitions and positive entropy is  an example of direct application of Theorem~\ref{thm:resultinmeasure} together with Theorem~\ref{thm:mainweights}. If $\mathcal H$ is any sequence of partitions of positive entropy $h$ in measure, then
\[ \lim_{n\to\infty} \frac{L_n(x,\mathcal{SB},\mathcal H)}{n\log L_n(x,\mathcal{SB},\mathcal{H})}=\frac{6h}{\pi^2}\quad\text{in measure} \]
and
\[ \lim_{n\to\infty}\frac{L_n(x,\mathcal H,\mathcal{SB})}{n/\log n}=\frac{\pi^2}{6h}\quad\text{in measure}, \]
with respect to the Lebesgue measure.

\section{Basic properties of sequences of partitions}
\label{sec:bal-Lochs}

This section gathers basic results which describe the  behavior of log-balanced sequences of partitions and of  Lochs' index. 
The first one, namely  Proposition \ref{prop:balanced-norm}, shows that the measures of intervals of  a log-balanced sequence of partitions tend to zero as their depths tend to infinity. This has two consequences: first, points have zero measure (Proposition \ref{prop:atom}); second,  Lochs' index is a.e. finite (Proposition \ref{prop:Lochsisfinite}).
Proposition \ref{prop:equivalent} then shows that if a sequence of  partitions is log-balanced    with respect to a measure  $\lambda$, so is it with respect to any measure equivalent to $\lambda$. This property is relevant for the study of number-theoretic fibred systems (see Section \ref{subsec:fibred}).
Proposition~\ref{prop:nonmelangeepart1} shows that any a.e.\ log-balanced sequence of partitions admits some nondecreasing weight function.
Finally, Proposition \ref{prop:polynomially} proves that a sequence of partitions with a  sub-exponential number of intervals at each depth has zero entropy. In particular, if it is log-balanced, its weight function is in $o(n)$. This is the case for the Farey and the three-distance sequences of partitions.

The following proposition shows that the norm of the partitions of a log-balanced sequence tend to zero.

\begin{proposition}\label{prop:balanced-norm}
Let $\mathcal P=\{P_n\}_{n\in\mathbb N}$ be a sequence of partitions. If $\mathcal P$ is log-balanced in measure (or even a.e.) with respect to some measure $\lambda$, then
\[ \Vert P_n\Vert_\lambda:=\sup_{I\in P_n}\lambda(I)\to 0\quad\text{as }n\to\infty.\]
\end{proposition}
\begin{proof}

Suppose otherwise $\Vert P_n\Vert_\lambda\not\to 0$
as $n\to\infty$. Thus, there is some $c>0$ and some increasing sequence $\{n_k\}_{k\in\mathbb N}$ such that $\Vert P_{n_k}\Vert_\lambda\to c$ as $k\to\infty$. Let $J_{k}\in P_{n_k}$ such that $\lambda(J_{k}) > c/2$, for large enough $k$.

Let $f$ be a weight function of $\mathcal P$ in measure (or even a.e.) with respect to $\lambda$. Since $f(n_k)\to+\infty$ as $k\to\infty$, we have that $1+\log(c/2)/f(n_k)\geq 1/2$ for $k$ large enough. Hence, as for each $x\in J_{k} $, we have $J_{k} = I_{n_k}(x)$, then
\begin{align*}
\frac c2< \lambda(J_{k})&\le \lambda\left(\left\{x\in[0,1]:\, 1 +\frac{\log \lambda(I_{n_k}(x))}{f(n_k)}> 1 +\frac{\log (c/2)}{f(n_k)} \right\}\right)\\
& \le \lambda\left(\left\{x\in[0,1]:\, \left|1 +\frac{\log \lambda(I_{n_k}(x))}{f(n_k)}\right|>\frac 12\right\}\right),
\end{align*}
for large enough $k$. This shows that  $-\log(\lambda(I_n(x))/f(n)$ does not converge towards $1$ as $n\to\infty$ in measure and thus neither does a.e. This contradicts that $f$ is a weight function of $\mathcal P$.\end{proof}

The next propositions show that the property of  being log-balanced with respect to $\lambda$
implies that $\lambda$ charges no point and moreover that  the measure
of the intervals  of the  partitions $P_n$ tend to $0$ as $n\to\infty$.  Also,   equivalent measures  yield the same log-balanced sequences and weight functions.

\begin{proposition}\label{prop:atom}
Let $\mathcal P=\{P_n\}_{n\in\mathbb N}$ a sequence of partitions. If $\mathcal P$ is log-balanced in measure (or even a.e.) with respect to some measure $\lambda$, then $\lambda(\{x\})=0$ for each $x\in[0,1]$.
\end{proposition}
\begin{proof} On the one hand, $\lambda(E)=0$ by  Definition~
\ref{def:balanced}. On the other hand, if $x\in[0,1]\setminus E$ and $n\in\mathbb N$, then $\lambda(\{x\})\leq\lambda(I_n(x))\leq\Vert P_n\Vert_{\lambda}$ and thus $\lambda(\{x\})=0$ by Proposition~\ref{prop:balanced-norm}.
\end{proof}

In  our setting of log-balanced sequences  of partitions,  Lochs' index takes finite values a.e. such as stated below,  where  we recall that $E^i$ stands for the set of endpoints of the partition ${\mathcal P}^i $, $i=1,2$. 

\begin{proposition}\label{prop:Lochsisfinite}
Let $\mathcal P^1$ and $\mathcal P^2$ be sequences of partitions. If $\mathcal P^2$ is log-balanced a.e.\ (or in measure) with respect to some measure $\lambda$ and $\lambda(E^1)=0$, then the Lochs index $L_n(x,\mathcal P^1,\mathcal P^2)$ is finite a.e.\ with respect to $\lambda$ for each $n\in\mathbb N$.
\end{proposition}
\begin{proof} Let $n\in\mathbb N$. Since $\mathcal P^2$ is log-balanced a.e.\ (or in measure), $\lambda(E^2)=0$. Let $J$ be the union of all the intervals $I_n^1$ in $P_n^1$ such that $\lambda(I_n)=0$. Clearly, $\lambda(J)=0$. Hence, it suffices to prove that $L_n(x,\mathcal P^1,\mathcal P^2)$ is finite for every $x\in[0,1]\setminus(E^1\cup E^2\cup J)$. For that purpose, let $x\in[0,1]\setminus(E^1\cup E^2\cup J)$. By Proposition~\ref{prop:balanced-norm}, $\lambda(I_\ell^2(x))\to 0$ as $\ell\to\infty$. As $\lambda(I_n^1(x))>0$, there must be some $\ell\in\mathbb N$ such that $I_n^1(x)\nsubseteq I_\ell^2(x)$  and thus $L_n(x,\mathcal P_1,\mathcal P^2)$ is finite.
\end{proof}

\begin{proposition}\label{prop:equivalent}
Let $\mathcal P=\{P_n\}_{n\in\mathbb N}$ be a sequence of partitions that is log-balanced a.e.\ (resp.\ in measure) with respect to some measure $\lambda$.
If $\mu$ is another Borel probability measure equivalent to $\lambda$ (i.e., both $d\mu/d\lambda$ and $d\lambda/d\mu$ are bounded), then $\mathcal P$ is also log-balanced a.e.\ (resp.\ in measure) with respect to $\mu$ with the same weight function.
\end{proposition}
\begin{proof} By  Proposition~\ref{prop:balanced-norm},
$\mu(I_n(x))\to 0$ as $n\to\infty$ for every $x\in[0,1]\setminus E$. Now the proof can be completed along the same lines as in the proof of \cite[Lemma 2.4]{Dajani} in order to prove that
$\lim_{n\to\infty}\frac{\log  \lambda(I_n(x))}{\log \mu(I_n(x))}=1$
a.e.
\end{proof}

For the validity of Theorem~\ref{thm:resultae}, we assume that the weight $f_2$ of $\mathcal P^2$ is nondecreasing. %
The following lemma shows that this assumption does not per se rules out any particular a.e.\ log-balanced self-refining sequence $\mathcal P^2$.

\begin{proposition}
\label{prop:nonmelangeepart1}
Every self-refining a.e.\ log-balanced sequence of partitions $\mathcal P$ with respect to some measure $\lambda$ admits some nondecreasing a.e.\ weight function with respect to $\lambda$.
\end{proposition}
\begin{proof} Let $f$ be an a.e.\ weight function of $\mathcal P$. Let $x\in[0,1]\setminus E$ for which $-\log \lambda(I_n(x))/f(n)\to 1$ as $n\to\infty$. Let $g(n)=-\log\lambda(I_n(x))$. By construction, $g(n)/f(n)\to 1$ as $n\to\infty$. Hence, as $f$ is an a.e.\ weight function of $\mathcal P$, so is $g$. Moreover, since $\mathcal P$ is self-refining, the definition of $g$ implies that $g$ is nondecreasing.
\end{proof}

The next result shows that a sequence of partitions with polynomially many intervals at each depth has zero entropy a.e.

\begin{proposition}\label{prop:polynomially}
Let $\mathcal P=\{P_n\}_{n\in\mathbb N}$ be a sequence of partitions and let $\lambda$ be a  Borel probability measure on $[0,1]$. Suppose that the number $\#(P_n)$ of %
intervals at depth $n$ is such that $\log\#(P_n)=o(n)$ as $n\to\infty$.  Then, $\mathcal P$ has zero entropy a.e.
\end{proposition}
\begin{proof} Notice that  $\log\#(P_n)=o(n)$ implies $\limsup_{n\to\infty}\#(P_n)^{1/n}=1$. Thus, by the root test, $\sum_{n=0}^\infty\#(P_n) e^{-\epsilon n}<\infty$ for each $\epsilon>0$.

Let $h_n(x):=-\log\lambda(I_n(x))/n$. And, for each $\epsilon>0$ and $n\in\mathbb N$, let
\[A_{n,\epsilon}=\{x\in [0,1]\setminus E: h_n(x)\geq\epsilon\}=\{x\in [0,1]\setminus E: \lambda(I_n(x))\le e^{-\epsilon n} \}\ .
\]
Therefore,
\[ \lambda(A_{n,\epsilon})\le \sum_{I_n\in\mathcal P_n:\,I_n\cap A_{n,\epsilon}\neq \emptyset} \lambda(I_n)\le\#(P_n) e^{-\epsilon n}\ .
\]
Since $\sum_{n=0}^\infty\#(P_n) e^{-\epsilon n}<\infty$, the Borel--Cantelli Lemma implies
\[
\lambda(\{x:\,x\in A_{n,\epsilon}\ \text{i.o.}\})=0.
\]
Thus, $x\notin A_{n,\epsilon}$ for $n$ large enough a.e., that is, 
$\lim_{n\to \infty}h_n(x)=0$ a.e.\end{proof}

\section{Main examples}\label{sec:balanced}

In this section, we introduce our main examples of sequences of partitions and discuss their log-balancedness. Throughout this section we prove Theorem~\ref{thm:mainweights}. %

This section is organized as follows. In Section \ref{subsec:anyf}, we show that one can give an explicit  log-balanced  sequence of partitions with weight $f$ for any given function $f$ that tends to infinity. In Section~\ref{subsec:num}, we make explicit the connection between sequences of partitions and numeration systems. In Section \ref{subsec:fibred}, we build log-balanced sequences of partitions from fibred systems. In Section~\ref{subsec:fp}, we consider fibred numeration systems with positive entropy, which includes the sequences of partitions corresponding to the binary numeration system, the beta-numeration, and continued fraction expansions. In Section~\ref{subsec:fz}, we introduce the Stern-Brocot sequence of partitions, which is induced by a fibred system of zero entropy. In Section~\ref{subsec:nfz}, we present some examples of zero-entropy sequences of partitions not induced by fibred systems, namely the Farey sequence of partitions. Lastly, in Section \ref{subsec:3D}, we discuss a sequence of partitions associated with the three distance theorem.

\subsection{A realization  result} \label{subsec:anyf}

Next proposition shows that for any $f:\mathbb{N}\to \mathbb{R}_{\ge 0}$ so that $f(n)\to\infty$ as $n\to\infty$ there exists a sequence of partitions with $f$ as an a.e.\ weight function with respect to the Lebesgue measure. 

\begin{proposition}
Let $f:\mathbb{N}\to \mathbb{R}_{\ge 0}$, with $\lim_{n\to \infty} f(n)= \infty$.  Then, there exists an a.e.\ log-balanced sequence  of partitions $\mathcal{P}=\{P_n\}_{n\in\mathbb N}$ with weight function $f$, with respect to the Lebesgue measure. Moreover,  if $f$ is nondecreasing, then $\mathcal{P}$ is self-refining. Furthermore,  if $n\mapsto \lfloor e^{f(n)}\rfloor$ is (strictly) increasing, then $\mathcal P$ is strictly self-refining.
\end{proposition}

\begin{proof}
We build $\mathcal P$ by giving for each $n\in\mathbb N$ the set endpoints of the intervals of $P_n$. For that purpose, we take as a reference the sets $E_k^{\mathcal B}=\{a/2^k:\,a\in\mathbb N\text{ and }0\leq a\leq 2^k\}$ of endpoints of the binary partition at depth $k$.

Fix $n\in \mathbb{N}$ and consider the only integer $k\ge 0$ so that $2^k\le e^{f(n)}<2^{k+1}$.
These relations
allows one to  perform a  change of scale  %
between indices  of depths $k$ and $n$. Let $P_n$ be the topological partition of $[0,1]$ whose endpoint set is  
  \[ E_n = E_k^{\mathcal{B}}\cup \{a/2^{k+1}:\,a=2i+1\text{ and }0\le i \le \lfloor e^{f(n)}\rfloor-2^k-1, i\in \mathbb{N} \}.
 \]
Notice that $E_k^{\mathcal{B}} \subseteq E_n \subsetneq E_{k+1}^{\mathcal{B}}$ because $E_n$ is $E_{k}^{\mathcal{B}}$ together with the leftmost $\lfloor e^{f(n)}\rfloor-2^k$ endpoints of $E_{k+1}^{\mathcal{B}}$. Hence, if 
$n\mapsto \lfloor e^{f(n)}\rfloor$ is nondecreasing,  then the sequence
of partitions $\{P_n\}$ is  self-refining, and if $n\mapsto \lfloor e^{f(n)}\rfloor$ is  increasing, then $P_n\subsetneq P_{n+1}$. 

 Let $x\notin E^{\mathcal{B}}:=\bigcup_{k=0}^\infty E^{\mathcal B}_k$. Then, by construction, $2^{-k-1}\le |I_n(x)|\le 2^{-k}$, that is, 
 \[ \frac 1{2e^{f(n)}}\le\vert I_n(x)\vert<\frac 2{e^{f(n)}}.\]
The fact that $f(n)\to \infty$ as $n\to \infty$ implies that
 \[
  \lim_{n\to \infty}\frac{-\log |I_n(x)|}{f(n)}=1 \quad  a.e. \qedhere\]

 \end{proof}

\begin{rmk} Notice that, in the above proposition, if $f(n)/n\to 0$ (resp.\ $f(n)/n\to\infty$) as $n\to\infty$, then the sequence of partitions $\mathcal P$ has $0$ (resp.\ $\infty$) a.e.\ entropy with respect to the Lebesgue measure.
\end{rmk}

\subsection{From  sequences of partitions to numeration systems}\label{subsec:num}
 
By numeration system given by partitions, we mean  a system of representation of numbers using sequences  of labels, and even of digits in most of the  cases,   provided  by a  sequence of  partitions $\mathcal{P}= \{P_n\}_{n \in \mathbb N}$  endowed with labels.
This involves in particular the most   classical numerations for  real numbers, such as   $b$-ary representations, as well
as  representations based on continued fractions.

A  {\em numeration system    by  a log-balanced sequence of  partitions} over the
(at most countable) alphabet $\mathcal{A}$,  denoted as   ${\mathcal N}=({\mathcal P}, \rho)$, is 
  defined by 
a  strictly  self-refining sequence of partitions
$\mathcal{P} = \{P_n\}_{n\in\mathbb N}$
with  $ P_0 = \{(0,1)\}$, assumed to be log-balanced with respect to the Lebesgue measure, 
together with a sequence of labelling functions $\rho=\{\rho_n\}_{n \geq 1}$,  with  $\rho_n: P_n\to\mathcal A$ for all $n$, that  satisfies the following condition.

\begin{enumerate}[(H)]

\item For each $I\in\mathcal \mathcal{P}_n$ and each distinct $J,J'\in P_{n+1}$ such that $J,J'\subset I$, $\rho_{n+1}(J)\neq\rho_{n+1}(J')$. In other words, the restriction  of the labelling  map $\rho_{n+1}$ is injective  when restricted to the set $P_{n+1}|_I$ of all intervals $J\in P_{n+1}$ contained in
$I\in P_{n}$.
\end{enumerate}

In particular, by Proposition \ref{prop:balanced-norm}, one has $\|P_n\| = \sup \{|I|:\,I\in P_n\}$ tends to $0$ as $n\to\infty$. The sequence of  partitions  is thus {\em generating}
 (i.e., for a.e. $x,y \in [0,1]
 $, if $x\neq y$,  then there exist $n\in\mathbb N$, $P,Q\in P_n$, $P\neq Q$, such that  $x \in P$ and
 $y \in Q$.)

In full generality,  according e.g. to the formalism developed in \cite{BBLT:06}, 
a  {\em numeration system}   is a triple $(X,\mathcal{A},\varphi)$, where $X$ is
a set (the set of elements to be represented), $\mathcal{A}$ a finite or countable set (the alphabet of the representation), and $\varphi$ an injective map
$\varphi\colon X\rightarrow \mathcal{A}^{{\mathbb N }}$.  Accordingly, given a   sequence  of  partitions ${\mathcal P}$ that satisfies the above conditions,  we want to define a map that associates with every $x$ in $[0,1]$ the sequence of labels of intervals to which $x$ belongs. 
Such a map is well-defined on points that do not belong to  the set $E$ of endpoints of the intervals of the partitions.    In any case,  the set $E$  is a  countable set\footnote{Of course, it would be possible to define (two) finite expansions for the elements in $E$, but we choose to concentrate just on $[0,1]\setminus E$, the set corresponding to infinite expansions.}. Hence here  $X$ is equal to $[0,1]\setminus E$ and this gives a  coding  map $\varphi$   that  associates with an element of $[0,1]\setminus {E}$ the  sequence of labels of the intervals to which it belongs:
\[\varphi: [0,1] \setminus E\rightarrow \mathcal{A}^{\mathbb{N}}, \ 
 x \mapsto \rho_1(I_1(x))\rho_2(I_2(x))\ldots
\text{ where } x\in I_n(x) \text{  for all } n.\]

Let us check that the map $\varphi$ is  injective. Let $x,y\in X$ and suppose $\varphi(x)=\varphi(y)$. Observe that  $x$ and $y$ both  belong to
$[0,1]$.  Let us  prove by induction that $x$ and $y$ belong to the same interval $I_n\in P_n$ for all $n$. Assume that  the induction property holds for some positive  $n$.   The points  $x,y$ thus  belong to the same interval $I_{n}\in P_{n}$,   and by Condition (H),  $x,y$ belong to   the same $I_{n+1}\in P_{n+1}$, which ends the induction proof. Thus, $x,y\in\bigcap_{n\geq 1}I_n$. Since $\{I_n\}$ is a nested sequence of intervals whose length tends to $0$, $\bigcap_{n\geq 1}I_n$ has only one element. Hence, $x=y$.
We thus get a numeration system as defined  above.
In other words, by assumption (H), we have    for any $x \in [0,1]\setminus E$ that
$\{x\}%
=\bigcap_{n\ge 1}\overline{ I_n(x)}$
and the sequence of labels $ \{\rho_n(I_n(x))\}_{n\geq 1}$  encodes $x$ univocally. 
By definition, each interval of depth $n$ in the partition $P_n$  gathers all numbers in $[0,1]\setminus E$ whose representations, as sequences of labels, coincide until depth $n$.

\begin{rmk} Examples of  labelled  sequences  of partitions are given below.
When each interval in $P_n$ is divided into proper subintervals in $P_{n+1}$,  the labels $\rho_{n+1}$  usually produce digits that will be used to provide suitable expansions  of the real numbers of the unit interval. This is for instance the case of the binary numeration (see Section \ref{subsec:fp}). For other sequences of partitions, such as the Farey one in Section \ref{subsec:nfz}, the labelling $\rho_n$ might seem to be less relevant at first view. However, we will see that it is connected to the coding of dynamical systems in the context of characteristic Sturmian words, as developed in Section~\ref{subsec:sturm}.
\end{rmk} 

\subsection{Fibred  systems}\label{subsec:fibred}

Often, a sequence of partitions, which allows for the definition of a numeration system as described in Section \ref{subsec:num}, can be defined in dynamical terms. A particularly relevant  framework in this setting  is the one of fibred systems as introduced in  \cite{Schweiger:95}, and in   \cite{Dajani,DaField:01}   for the notion of a  number-theoretic fibred system.

In this paper, we say that the pair $([0,1], T)$ is a   {\em fibred system} if the 
transformation $T\colon [0,1]\to [0,1]$ is such that  there exist a finite or countable
set ${\mathcal A}$ and a topological  partition  ${ P}= \{I_a \}_{{a\in {\mathcal A}}}$ of $[0,1]$ for which the
restriction $T_a$ of $T$ to $I_a$ is injective and  continuous,  for each $a\in {\mathcal A}$.

We associate with a fibred system $([0,1], T)$ a sequence of partitions in the usual dynamical way. Let $T^{-1}{P}$ denote the partition $\{T^{-1} I_a\}_{a\in\mathcal A}$. We define, for any $n\ge 1$, the partition ${P}_n$  as  the join  partition  $\bigvee_{i=0}^{n-1}  T^{-j} { P}$
made up of the sets of the form 
$\bigcap _{j=0} ^{n-1} T^{-j} I_{a_j}$, for all choices of $a_0,\ldots,a_{n-1}\in\mathcal A$, where each $I_{a_j}\in  {{P}}$. For $n=0$,  we define $P_0=\{(0,1)\}$. %
Now, we define a sequence of labelling maps $\{\rho_n\}_{n\ge 1}$ as follows:
\[  \rho_n \left(\bigcap_{j=0}^{n-1}  T^{-j} I_{a_j}\right):=a_{n-1} \quad\text{for all choices of }a_0,\ldots,a_{n-1}\in\mathcal A. \]
Note that  $\rho_n= \rho_1 \circ T^{n-1}$ for each $n$.
If, moreover, ${\mathcal P}=\{P_n\}_{n\in\mathbb N}$ is log-balanced a.e.\ or in measure with respect to the Lebesgue measure and it is strictly self-refining, then $\mathcal P$ satisfies the  conditions given in Section \ref{subsec:num} and thus provides a numeration system  $({\mathcal P}, \rho)$. This numeration system associated with the map $T$ together with the topological partition $P$ (whose elements are indexed by the alphabet $\mathcal A$) is said to be, here, a  {\em fibred numeration system}.
The injectivity of the restriction $T_a$ of $T$ to $I_a$ is consistent with the injectivity requirement of condition (H).

Fibred systems as above defined are particular cases of the more general fibred systems developed in  \cite{BBLT:06}. For a fibred system $([0,1],T)$ as above,  the corresponding set $X$ and the corresponding partition $\{X_a\}_{a\in \mathcal{A}}$ of $X$ in the definition of fibred system given in Definition 2.3 of \cite{BBLT:06} are given by:
\[ X=X(T,P):=\bigcap_{j=0}^\infty T^{-j}\bigcup_{a\in \mathcal{A}} I_a\quad\text{and}\quad X_a=X\cap I_a\ \text{for each }a\in\mathcal A. \] Clearly, $T(X)\subseteq X$ and the restriction $T_a$ of $T$ on $X_a$ is injective. In our context, the set $X$ equals the set $[0,1]\setminus E$ where $E$ is the set of endpoints of $\mathcal P$.

A fibred system  $([0,1],T)$ admits \emph{an invariant and ergodic Borel probability measure} $\lambda$ if, for  every Borel set $B$, the following assertions hold: (i) $\lambda (T^{-1}(B))=\lambda(B)$ and (ii) $T^{-1}(B) = B$ implies that $B$ is a set of measure $0$ or $1$.
We then denote it as $([0,1],T,\lambda)$.
According to \cite[Definition 2.2]{Dajani}, a fibred system
is called a {\em number-theoretic fibred system}
if $T$ has an  ergodic and invariant probability measure which is 
equivalent to the Lebesgue measure.
A fibred system with an invariant and ergodic probability measure
$\lambda$ has then  a well-defined entropy (see {\cite[Chap.\ 8]{pollicott-yuri-1998}}).
The following theorem is a consequence of the Shannon--McMillan--Breiman Theorem (see {\cite[p. 134]{pollicott-yuri-1998}}).

\begin{theorem}\label{teo:ShannonMcMillanBreiman}
    Let  $ ([0,1],T,\lambda)$ be a fibred number-theoretic system (endowed with an ergodic invariant measure $\lambda$  of the  probability space $[0,1]$).
    Let $\mathcal{P}=\{P_n\}_{n \in \mathbb{N}}$ be the  sequence of partitions associated with  this fibred system.
   If $$H_1(\mathcal{P},\lambda)=-\sum _{I\in P_1} \lambda(I)  \log  
\lambda( I )<\infty,$$  then for $\lambda$-almost every $x$, the following limits exist and satisfy
    \[
    h = \lim_{n\to \infty} - \frac{\log \lambda (I_{n}(x)) }{n}
    =-\lim_{n\to \infty}\frac{1}{n}\sum _{I\in P_n} \lambda(I)  \log \lambda(I)  \,.
    \]
    In particular, the sequence of partitions ${\mathcal P}$
    is log-balanced a.e.\ with respect to $\lambda$ and the map $f:n\mapsto hn$ is a corresponding weight function. %
\end{theorem}
  
Under the above assumption, we thus get, from a fibred number-theoretic  system, a log-balanced  sequence of partitions, as described in Section \ref{subsec:num}. By Proposition \ref{prop:equivalent},  if $\mu$ is another probability measure equivalent to $\lambda$, then $\mathcal P$ is also log-balanced a.e.  for $\mu$ with $f$ as weight function.

\subsection{Fibred systems with positive entropy}\label{subsec:fp}

Two classical instances of ergodic  measure  preserving  fibred  systems  with  positive  entropy  are  the binary numeration system (in general, the beta-numeration) and the Gauss map producing continued fractions expansions. We describe them below.

\subsubsection{The binary numeration system}\label{subsubsec:binary} 

The binary numeration system is the fibred number-theoretic system associated with the map $T:x\mapsto 2x \mod 1$ and the partition $P=\{(0,1/2),(1/2,1)\}$ with respect to the Lebesgue measure. Its entropy for the Lebesgue measure is $\log 2$.
It associated sequence of partitions is the {\em binary  sequence of partitions} defined as  $\mathcal B=\{B_n\}_{n\in\mathbb N}$, where
\[ B_n=\left\{ \left(0,\frac 1{2^n}\right),\left(\frac 1{2^n},\frac 2{2^n}\right), \ldots, \left(\frac{2^n-1}{2^n},1\right)\right\}.\]
Hence, the endpoints of the intervals in $B_n$ are $E_n^{\mathcal B}$ are the points $a/2^n$ for $a=0,1,\ldots,2^n$. The {\em binary  sequence of partitions}  $\mathcal B$ is a.e.\ log-balanced with weight function $f_{\mathcal  B}(n)=(\log 2)n$ with respect to the Lebesgue measure.

\subsubsection{The decimal numeration system.} \label{subsubsec:decimal}  The decimal numeration system is the fibred number-theoretic system associated with the map $T:x\mapsto 10x\mod 1$ with the partition $P=\{(a/10,(a+1)/10)\}_{a=0,\ldots,9}$ with respect to the Lebesgue measure. We denote by $\mathcal D=\{D_n\}_{n\in\mathbb N}$ the {\em decimal sequence of partitions} associated with this fibred systems.
It admits as an a.e.\ weight function $f_{\mathcal D}(n)=(\log 10)n$ with respect to the Lebesgue measure.

\subsubsection{Beta-numeration sequence of  partitions}\label{subsubsec:beta} The {\em beta-numeration  sequence of partitions}  $\mathcal{ BE}$ is the   sequence  of partitions associated with  the 
the beta-numeration, i.e.,       the fibred system associated with the map $x\mapsto \beta x \mod 1$, 
 where $\beta$ is a  given  real number  with $\beta>1$.  Its entropy for the Lebesgue measure is $\log \beta $.   It is  a.e.\ log-balanced with weight function $f_{\mathcal  B}(n)=(\log \beta)n$. This  numeration   has been introduced in  \cite{Renyi}.
For more on the length of the fundamental intervals, see   e.g.  \cite{FW:12}.  When $\beta=2$, one recovers the binary 
numeration and when  $\beta=10$, the decimal one. 

\subsubsection{Continued fraction   sequence of partitions}\label{subsubsec:cf} 

The continued fraction expansions correspond to the fibred system associated with the Gauss map $T_G$ such that 
\[T_G(x)=\{1/x\}\quad\text{if } x\neq 0\qquad\text{and}\qquad T_G(0)=0, \]
where $\{x\}$ denotes the fractional part of $x$ and the partition $\{(1/(a+1),1/a)\}_{a\geq 1}$. The corresponding alphabet is the set of the positive integers. The labelling map $\rho_1$ is given by $\rho_1(x)=\lfloor 1/x\rfloor$ and $\rho_n(x)=\rho_1(T_G^{n-1}(x))$. We denote by $\mathcal{CF}$ the  sequence of partitions associated with this fibred system, called the {\em continued fraction   sequence of partitions}.

We use  the standard notation for continued fractions 
\[x=[0;a_1(x),a_2(x), a_3(x),\ldots]=
\cfrac{1}{a_1(x)+
\cfrac{1}{a_2(x)+
\cfrac{1}{a_3(x)+\dotsb
}}}
\]
 with the last partial quotient larger than 1 in the case of rationals numbers. As usual, the numerators and denominators of the $n$-th convergent $[0;a_1(x),a_2(x), \ldots, a_n(x)]$ are called continuants of the number $x$. They are denoted by  $p_n=p_n(x)$ and $q_n=q_n(x)$. For any $n\ge 1$, the following recurrence relations hold  
\begin{flalign*}
  &p_{n}=a_{n}p_{n-1}+p_{n-2},\  p_0=0,\ p_{-1}=1,\\
  &q_{n}=a_{n}q_{n-1}+q_{n-2},\  q_0=1,\ q_{-1}=0,
\end{flalign*}
together with the equality
\begin{equation}\label{eq:Teo2Khinchin}
  q_np_{n-1}-p_nq_{n-1}=(-1)^{n}.
\end{equation}

The fundamental interval $I_n^{\mathcal{CF}}(x)$ of depth $n$ associated with $x$ in the continued fraction expansion has as endpoints the fractions $p_n/q_n $ and 
 $(p_{n-1}+p_n)/(q_{n-1}+q_n)$.
 More precisely, by~\cite[p.\ 57]{khinchin}, one has
\begin{equation}
 I_n^{\mathcal{CF}}(x)
 =\begin{cases}
 \left(\dfrac{p_n}{q_n}, \dfrac{p_{n-1}+p_n}{q_{n-1}+q_n}\right)  & \text{if $n$ is even}\bigskip\\
 \left(\dfrac{p_{n-1}+p_n}{q_{n-1}+q_n}, \dfrac{p_n}{q_n}\right) &\text{if $n$ is odd}.
 \end{cases}
 \label{eq:interval-CF}
\end{equation}
In particular,
\begin{equation}\label{eq:length-CF-In(x)}
\left\vert I_n^{\mathcal{CF}}(x)\right\vert=\frac{1}{q_n (q_n + q_{n-1})}.
\end{equation}

Let the \emph{Gauss measure} be $d\mu=1/((1+x)\log 2)\,dx$. Since $T_G$ is an ergodic measure preserving transformation with respect to the Gauss measure (see, for instance, \cite[p. 106]{pollicott-yuri-1998}), the dynamical  system  $([0,1],T_G,\mu)$ is a fibred number-theoretic system. Moreover, as the Gauss measure is absolutely continuous with respect to the Lebesgue measure, it can be shown, by relying on Theorem~\ref{teo:ShannonMcMillanBreiman}, that the entropy of the  continued fraction  sequence of partitions  $\mathcal{CF}$ is $\pi^2/(6\log 2)$, 
with respect to both its invariant measure and the Lebesgue measure (see \cite[Cor. 4.1.28]{IK2002} for a detailed proof). 
It follows that the sequence of partitions  $\mathcal{CF}$  admits as an a.e.\ weight function
\begin{equation}\label{eq:f_CF}
  f_{\mathcal{CF}}(n)=\frac{\pi^2}{6\log 2}n
\end{equation}
with respect to the Lebesgue measure (and the Gauss measure as well).

An alternative derivation of the value of the a.e.\ entropy of $\mathcal{CF}$ follows from Khinchin--L\'evy's Theorem, which asserts that
\begin{equation}\label{eq:Levy}
\lim_{n\to \infty} \frac{\log q_n(x)}{n}=\frac{\pi^2}{12\log 2}\quad\text{ a.e.\ $x$}.
\end{equation}

We also recall  the  Borel--Bernstein Theorem (see \cite{BOrel:09,Bernstein:11}).
\begin{theorem}[Borel-Berstein]\label{theo:BB}
The set of real numbers $x$ such  that $a_n \leq k_n$  for  infinitely many $n$, where the $k_n$ are positive integers is of zero measure  if and only if $\sum_n 1/k_n =\infty$.

\end{theorem}

\subsection{A fibred numeration system with zero entropy: the Stern-Brocot system} \label{subsec:fz}

{\em  The Stern-Brocot  sequence of partitions}, denoted as  $\mathcal{SB}$,   is  the sequence of partitions associated with the  Farey fibred system whose map is the Farey map
\[ T_F=\begin{cases}
  \dfrac 1{1-x}-1&= \dfrac{x}{1-x}\quad \text{if }x\in[0,1/2]\bigskip\\
  \dfrac 1x-1&=\dfrac{1-x}{x}\quad \text{if } x\in[1/2,1]
\end{cases} \]
and the partition $P=\{(0,1/2),(1/2,1)\}$.
This sequence has been widely studied, see e.g. 
\cite{MZ:04}. Even if  the partition is nonuniform, considering logarithms of lengths  allows
the balance property in measure. 

The standard construction of the sequence $\mathcal{SB}=\{\text{\emph{SB}}_n\}_{n\in\mathbb N}$ is as follows (see Figure~\ref{fig:SB} for an illustration): we start with $\text{\emph{SB}}_0=\{(0/1,1/1)\}$, and
 for each $n\in\mathbb N$, 
$\text{\emph{SB}}_{n+1}$ arises from $\text{\emph{SB}}_n$ 
by dividing each interval $(p/q,p'/q')$ of $\text{\emph{SB}}_{n-1}$ into two subintervals by its mediant 
$(p+p')/(q+q')$.

\begin{figure}
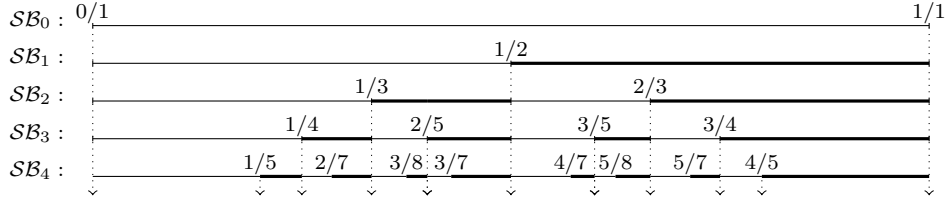

\begin{center}
\SternBrocot
\end{center}
\caption{The Stern-Brocot sequence of partitions}\label{fig:SB}
\end{figure}
We denote by $E^{\mathcal{SB}}_n$ the set of endpoints of the intervals in $\text{\emph{SB}}_n$.

Let $x\in[0,1]$  be an  irrational number and let  $\{p_m/q_m\}$ 
stand  for its sequence of convergents. We recall that the   convergents satisfy   the following
mediant-construction \cite[p.\ 14--15]{khinchin}: if $m$ is even, then
\begin{equation}\label{eq:mediant_even}
\frac{p_{m}}{q_{m}}\leq x\leq \frac{p_{m+1}}{q_{m+1}}= \frac{a_{m+1} p_m+p_{m-1}}{a_{m+1}q_m+q_{m-1}}\leq \cdots\leq\frac{p_{m-1}+p_n}{q_{m-1}+q_m}\leq\frac{p_{m-1}}{q_{m-1}},
\end{equation}
whereas if $m$ is odd, then
\begin{equation} \label{eq:mediant}
\frac{p_{m-1}}{q_{m-1}} \leq \frac{p_{m-1}+ p_n} {q_{m-1}+q_m} \leq \cdots \leq \frac{a_{m+1} p_m+p_{m-1}}{a_{m+1}q_m+q_{m-1}}=\frac{p_{m+1}}{q_{m+1}} \leq  x \leq \frac{p_{m}}{q_{m}}.
\end{equation}

The following result is well known (see~\cite[Lemma~1]{MoZhi}).

\begin{proposition}\label{prop:SternBrocotSB}
Let $x\in[0,1]$ be an irrational and let $n\in\mathbb N$. Then, the interval in $SB_n$ containing $x$ is
\[ I_n^{\mathcal{SB}}(x)=
\begin{cases}
  \left( \dfrac{p_m}{q_m}, \dfrac{(r+1) p_m +p_{m-1}}{(r+1) q_m +q_{m-1}} \right) & \text{if $m$ is even}\bigskip\\
  \left( \dfrac{(r+1) p_m +p_{m-1}}{(r+1) q_m +q_{m-1}}, \dfrac{p_m}{q_m} \right)   & \text{if $m$ is odd,}
\end{cases} \]
where $m:=m(x,n)$ and $r:=r(x,n)$ are the unique integers such that
\[ \sum_{i=1}^m a_i\leq n<\sum_{i=1}^{m+1}a_i\quad\text{and}\quad r=n-\sum_{i=1}^m a_i. \]
As a consequence, one has
\[
\left\vert I_{n}^{\mathcal{SB}} (x)\right\vert=\frac 1{((r+1)q_m+q_{m-1})q_m}. \]
\end{proposition}

The next two results show that, with respect to the Lebesgue measure, the Stern-Brocot sequence of partitions is log-balanced in measure but not a.e.

\begin{proposition}\label{prop:fSB} The Stern-Brocot sequence of partitions admits as weight function in measure with respect to the Lebesgue measure
\[
f_{\mathcal{SB}}(n):=\frac{\pi^2}6\frac{n}{\log n}\quad\text{for each }n\geq 2.
\]
\end{proposition}
\begin{proof} Let $x\in[0,1]$ an irrational number and let $n\in\mathbb N$.  Let $m$ and $r$  be as in Proposition~\ref{prop:SternBrocotSB}, i.e.,
 $\sum_{i=1}^m a_i\leq n<\sum_{i=1}^{m+1}a_i$ and $r=n-\sum_{i=1}^m a_i$. Hence,
\[
\frac 1{2(r+1)q_m^2}\le\left\vert I_{n}^{\mathcal{SB}} (x)\right\vert\le \frac 1{(r+1)q_m^2}.
\]
Therefore, one has 
\begin{equation}\label{eq:-logSB}
-\frac{\log\vert  I_n^{\mathcal{SB}} (x)\vert}{n/\log n}=\frac{\log(r+1)}{n/\log n} + 2\frac{\log q_m}{n/\log n}+O\left(\frac{\log n}n\right)\quad\text{as }n\to\infty.
\end{equation}
Since $r\le n$, the first term in the above equation tends to zero when $n\to \infty$ a.e. We write
\begin{equation}\label{eq:three-factors} \frac{\log q_m}{n/\log n}=\frac{\log q_m}{m}\frac{\log n}{\log m}\frac{m\log m}{n}.
\end{equation}
Notice that, by~\eqref{eq:Levy}, the first factor $\frac{\log q_m}{m}$ on the above right-hand side tends to $\frac{\pi^2}{12\log 2}$ as $m\to\infty$ a.e. Let us prove that
\begin{equation}\label{eq:log n-sim-log m} \lim_{n\to\infty} \frac{\log n}{\log m}=1\quad\text{a.e.\ }x. \end{equation}
(We recall that $m$ depends on $x$ and $n$.)
By~\cite[Corollary 1]{DiamondVaaler86}, for a.e.\ $x$ and large enough $m$ (depending on $x$),
\[ \sum_{i=1}^m a_i\leq\frac{1+o(1)}{\log 2}m\log m+\max_{1\leq i\leq m}a_i. \]
Fix $\epsilon>0$. 
By Borel--Bernstein Theorem (see Theorem \ref{theo:BB}), for a.e.\ $x$ and for $m$  large enough (depending on $x$), one has 
\[ a_m\leq m(\log m)^{1+\epsilon}. \]
For a.e.\ $x$ and $n$ large enough (depending on $x$), this yields
\begin{align*}
  m\leq n<\sum_{i=1}^{m+1}a_i&\leq\left(\frac{1+o(1)}{\log 2}\log(m+1)+(\log(m+1))^{1+\epsilon}\right)(m+1)\\
 &=O(m(\log m)^{1+\epsilon}).
\end{align*}
Hence, $\log n=\log m+O(\log\log m)$ as $n\to\infty$, a.e.\ $x$. This proves~\eqref{eq:log n-sim-log m}.

In \cite[p.\ 377]{khinchin-partial-sums}, Khinchin proved that  $\frac{\log 2} {m\log m} \sum_{i=1}^m a_i$ tends to $1$ in measure as $m\to \infty$. Since, by definition, $\sum_{i=1}^m a_i\leq n<\sum_{i=1}^{m+1}a_i$, this implies that
\begin{equation}\label{eq:mlogm/n-in-measure}
  \lim_{n\to\infty}\frac{m\log m}n=\log 2\quad\text{in measure}.
\end{equation}
With Equation \eqref{eq:three-factors}, the result follows.
\end{proof}

\begin{proposition} The Stern-Brocot sequence of partitions is not log-balanced a.e.\ with respect to the Lebesgue measure.
\end{proposition}
\begin{proof} If  the Stern-Brocot sequence of interval partitions $\mathcal{SB}$ were to admit an a.e.\ weight function $f$ with respect to the Lebesgue measure, then, because of Proposition~\ref{prop:fSB}, $f(n)/f_{\mathcal{SB}}(n)\to 1$ as $n\to\infty$. Thus, it suffices to prove that $f_{\mathcal{SB}}$ is not an a.e.\ weight function for $\mathcal{SB}$.

Let $S_m=S_m(x,n):=\sum_{i=1}^m a_i$.
Since $n\geq S_m$, by virtue of~\cite[Theorem 1]{Philipp88} we have that
\[ \liminf_{n\to\infty}\frac{m\log m}n\leq\liminf_{n\to\infty}\dfrac{m\log m}{S_m}=0\quad\text{a.e.\ }x. \]
Notice however that~\eqref{eq:mlogm/n-in-measure} holds. As a consequence, $\frac{m\log m}n$ does not converge as $n\to\infty$ a.e. As the first two factors on the right-hand side of \eqref{eq:three-factors} converge as $n\to\infty$ a.e., $\log q_m/(n/\log n)$ does not and, consequently, neither does the left-hand side of \eqref{eq:-logSB}. This proves that $f_{\mathcal{SB}}$ is not an a.e.\ weight function for $\mathcal{SB}$.
\end{proof}

The Stern-Brocot sequence of partitions has zero Shannon entropy with respect to the Lebesgue measure (see \cite{Aofa:20}). Moreover, the following also holds.

\begin{proposition}
The Stern-Brocot sequence of partitions has zero   entropy a.e.\ with respect to the Lebesgue measure.
\end{proposition}
\begin{proof} Letting $m$ be as in the proof of Proposition~\ref{prop:fSB} and combining \eqref{eq:-logSB}, \eqref{eq:three-factors}, and~\eqref{eq:Levy}, it follows that
\[ -\frac{\log\vert I_n^{\mathcal{SB}}(x)\vert}{n}=O\left(\frac mn\right)\quad\text{as }n\to\infty,\text{ a.e.\ }x.
\]
By~\cite[Corollary 1]{DiamondVaaler86},
for a.e.\ $x$, and for $n$ large enough (depending on $x$), $n\geq\sum_{i=1}^m a_i\geq m\log m$. The proposition follows.
\end{proof}

\subsection{The Farey  sequence  of  partitions   }\label{subsec:nfz}

In this section, we introduce two zero entropy sequences of partitions, namely the Farey sequence of partitions and  the closely related Sturmian  sequence of partitions. Both  are nonfibred  examples  having  zero entropy.
The fact  that they are indeed zero entropy follows from  Proposition \ref{prop:polynomially} because they have polynomially many intervals at each depth.

\subsubsection{The Farey sequence of partitions ${\mathcal F}$} The \emph{Farey sequence of partitions} $\mathcal F=\{F_n\}_{n\in\mathbb N}$ is defined as follows. Each of the partitions $F_n$ is determined by the set $E_n^{\mathcal F}$ of endpoints of its intervals, i.e., 
\[ E^{\mathcal F}_n:= \left\{ \frac p q:\,p,q \geq 1,\  \gcd(p,q)=1\text{ and }q \leq n+1\right\}. \]

The set of  endpoints $E^{\mathcal F}_n$ thus corresponds  to the  {\em Farey sequence of  order $n+1$}  which is the sequence of fractions $h/k$ with $(h,k)=1$ and $1 \leq h \leq k \leq n+1$, arranged in increasing order between 0 and 1, and studied e.g.
in \cite{Hall:70,KSS}.

Equivalently, $\mathcal F$ can be built recursively as follows. Let $ F_0=\{(0/1,1/1)\}$ and, for each $n\in\mathbb N$, let $F_{n+1}$ be the partition that arises from $F_n$ by dividing each of the intervals $(p/q,p'/q')$ such that $q+q'$ is at most $n+2$ into two subintervals by its mediant $(p+p')/(q+q')$, while keeping all the other intervals of $F_n$ unchanged.

The construction of the Farey sequence of partitions closely resembles the construction of the Stern-Brocot sequence of partitions given in Section~\ref{subsec:fz}. In fact, the only difference is that when constructing  the sequence of partitions $\mathcal{SB}$ we divide every interval $(a/c,b/d)$ of $\text{\emph{SB}}_n$ into two subintervals, whereas when constructing $\mathcal F$ only those intervals $(p/q,p'/q')$ of $F_n$ which satisfy $q+q'\leq n+2$ are divided. Compare Figures~\ref{fig:SB} and~\ref{fig:Farey}.

\begin{figure}
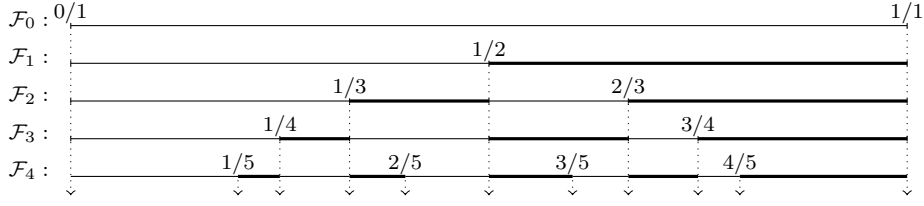

\begin{center}
\Farey
\caption{The Farey sequence of partitions}\label{fig:Farey}\end{center}
\end{figure}

By the definition of $E_n^{\mathcal F}$, it is clear that $F_n$ consists of $O(n^2)$ intervals and thus Proposition~\ref{prop:polynomially} implies that $\mathcal F$ has zero entropy. Notice that, instead, each partition $\text{\emph{SB}}_n$ consists of $2^n$ intervals.

Let us recall some basic facts regarding the Farey sequence.

\begin{theorem}[\cite{Hall:70,HW:08}]\label{thm:Hall}
Two irreducible fractions $p/q$ and $p'/q'$ in $[0,1]$ are consecutive endpoints in $\mathcal F_n$ if and only if $q\leq n+1$, $q'\leq n+1$, $|p'q-pq'|=1$ and $q+q'>n+1$.
\end{theorem}

The following result is probably folklore. However, given the lack of a suitable reference for it, we give a precise statement and a proof for the sake of completeness.

\begin{proposition}\label{prop:FareyFF}
Let $x\in[0,1]$ be an irrational and let $n\in\mathbb N$. Then, the interval in $F_n$ containing $x$ is
\[ I_n^{\mathcal F}(x)=
\begin{cases}
  \left( \dfrac{p_m}{q_m},\dfrac{(r+1) p_m +p_{m-1}}{(r+1) q_m +q_{m-1}} \right)&\text{if $m$ is even}\bigskip\\
  \left( \dfrac{(r+1) p_m +p_{m-1}}{(r+1) q_m +q_{m-1}}, \dfrac{p_m}{q_m} \right) &\text{if $m$ is odd,}
\end{cases} \]
where $m:=m(x,n)$ and $r:=r(x,n)$ are the unique integers such that
\begin{equation}\label{eq:m_and_r}
  (r+1) q_m+q_{m-1}\leq n+1<(r+2)q_m+q_{m-1},\ m\geq 0,\text{ and }0 \leq r < a_{m+1}.
\end{equation}
As a consequence,
\[ \vert I_n^{\mathcal F}(x)\vert=\frac{1}{ ((r+1) q_m +q_{m-1})q_m}. \]
\end{proposition}
\begin{proof} The uniqueness of $m$ follows from the fact that \eqref{eq:m_and_r} implies
\begin{equation*} q_m+q_{m-1}\leq n+1<(a_{m+1}+1)q_m+q_{m-1}=q_{m+1}+q_m.
\end{equation*}
The uniqueness of $m$ in turn implies the uniqueness of $r$.

Notice that \eqref{eq:mediant_even} or  \eqref{eq:mediant} (according to whether $m$ is even or odd) implies that  $x$ belongs to the interval whose endpoints are the quotients $p_m/q_m$ and $((r+1) p_m+p_{m-1})/((r+1)q_m+q_{m-1})$. As we are assuming that $q_m\leq (r+1)q_m+q_{m-1}\leq n+1$, both quotients are endpoints of $F_n$. Moreover, Theorem~\ref{thm:Hall} implies that these quotients are consecutive endpoints of $F_n$ because $q_m+((r+1)q_m+q_{m-1})=(r+2)q_m+q_{m-1}>n+1$ and 
\begin{align*}
    ((r+1)p_m+p_{m-1})q_m-((r+1)q_m+q_{m-1})p_m=p_{m-1}q_m-q_{m-1}p_m=1
\end{align*}
by \eqref{eq:Teo2Khinchin}.\end{proof}
 
\begin{rmk} Notice, by comparing Propositions~\ref{prop:SternBrocotSB} and \ref{prop:FareyFF}, the intervals  
 in the Stern-Brocot $\mathcal{SB}$ and in the Farey  $\mathcal{F}$ sequences of  partitions      that contain a given $x$ have the same form, however  they might occur  at different  indices for their depth. 
 
\end{rmk}

\begin{proposition}
The Farey sequence of partitions admits as a.e.\ weight function with respect to the Lebesgue measure
\[ f_{\mathcal F}(n)=2\log n\quad\text{for each }n\geq 1. \]
\end{proposition}
\begin{proof} Let $x\in[0,1]$ be  an irrational number. Let $m\geq 0$ as in Proposition~\ref{prop:FareyFF} (i.e.,
$(r+1)q_m+q_{m-1}\leq n+1<(r+2)q_m+q_{m-1}$ where $0\leq r<a_{m+1}$). Notice that $q_m\leq(r+1)q_m+q_{m-1}\leq n+1$. Notice also that $n+1<(r+2)q_m+q_{m-1}\leq(r+3)q_m$. Similarly, $n+1<(r+2)q_m+q_{m-1}\leq 2((r+1)q_m+q_{m-1})$. In this way, we proved that
\[ \frac 1{(n+1)^2}\leq\left\vert I_n^{\mathcal F}(x)\right\vert\leq\frac{2(r+3)}{(n+1)^2}. \]

It is well-known that $m=O(\log n)$ as $n\to\infty$ (since $2^{(m-1)/2}\leq q_m\leq n$). Hence, since Borel--Bernstein Theorem (Theorem~\ref{theo:BB}) ensures that $a_m=O(m(\log m)^2)$ as $m\to\infty$, for a.e., it follows that $r<a_{m+1}=O(\log n (\log\log n)^2)$ a.e. As a consequence,
\[ 
 -\log\vert I_n^{\mathcal{F}}(x)\vert= 2 \log (n+1)+ O(\log \log n)\quad\text{as $n\to\infty$,\quad a.e.\ }x,  \]
where the hidden constants in the $O$-term may depend on $x$. The result follows.\end{proof}

\subsubsection{The Farey sequence of partitions  and Sturmian words}\label{subsec:sturm}

The Farey sequence of partitions  has a  particular combinatorial meaning in symbolic dynamics  by  producing   prefixes of  so-called  characteristic Sturmian words. Indeed, 
given an irrational  real number $\alpha\in [0,1]$ consider the  Kronecker-Weyl sequence $ \{n\alpha\}_{ n \geq 1}$ (here $\{x\}$ is the fractional part of $x$).  Sturmian words  are obtained as 
 binary codings  of Kronecker-Weyl sequences, thus providing a numeration system as   introduced in Section \ref{subsec:num}  for  irrational numbers with digits in $\{0,1\}$. We describe it below.
 
 Let $\alpha$ be an irrational number in $[0,1]$. Consider the  two intervals $(0,1-\alpha)$ and $(1-\alpha,1)$. We define the sequence   $S(\alpha):=\{s_n(\alpha)\}_{ n \geq 1}$ in  $\{0,1\}^{\mathbb N}$  associated with  $\alpha$ as follows:
\[s_n(\alpha) =\begin{cases}

               0 & \hbox{ if } \{n\alpha\}\in (0,1- \alpha) \\
               1 & \hbox{ if }  \{n \alpha\} \in (1- \alpha,1).
              \end{cases}
\]
Since $\alpha$ is  assumed to be irrational,  observe that 
the sequence $\{n\alpha\}$  never takes the value $0$, nor $1-\alpha$.
We will use the notation
$S(\alpha)_{[1,n]}$ for the prefix $s_1 \cdots s_n$ of length $n$ of 
the sequence  $S(\alpha)$ (considered as an infinite word  over the alphabet $\{0,1\}$). The  sequence  $S(\alpha)$ is a so-called  {\em characteristic Sturmian word} (see  e.g. \cite[Chapter 2]{Lothaire}).
Sturmian words are among the most  studied  words in  word combinatorics and symbolic dynamics. 

  By \cite[Lemma 5]{VB:96},  irrational  numbers $\alpha$ that belong to a common interval of the partition $F_n$ have the same
 prefix  of length $n$ for the characteristic Sturmian word $s_n(\alpha)$,  while  these prefixes  differ if they belong to two
 distinct intervals of $F_n$.  Hence, with each interval  in the Farey partition $F_n$ of order $n$ 
is associated  the  prefix of length $n$  of some characteristic Sturmain word $S(\alpha)$.

More precisely, according to \cite{VB:96,Mignosi},  let  $\frac{p_1}{q_1}$  and $\frac{p_2}{q_2}$
be the endpoints of  an interval in $F_n$ with $\frac{p_1}{q_1} \neq 0$, $\frac{p_2}{q_2} \neq 1$ and $\frac{p_1}{q_1}<  \frac{p_2}{q_2}$.
Let $\alpha $ be an irrational number in $[\frac{p_1}{q_1}, \frac{p_2}{q_2}].$  If  $\alpha$ belongs to $[\frac{p_1}{q_1}, \frac{p_1+p_2}{q_1+q_2}]$, then $S(\alpha)_{[1,n+1]} = S(\alpha)_{[1,n]} 0$, and  if  $\alpha \in  [ \frac{p_1+p_2}{q_1+q_2},\frac{p_2}{q_2}]$,
then $S(\alpha)_{[1,n+1]} = S(\alpha)_{[1,n]} 1$. Moreover,  for all  irrational $ \alpha$ in $[\frac{p_1}{q_1}, \frac{p_2}{q_2}]$, $S(\alpha)]_{[1,n+1]}$ is a palindrome
 if and only if $q_1+q_2 =n+2$.
 This thus allows the definition of  a labelling  function   $\rho_n:F_n\to\{0,1\}$ such as introduced in 
 Section \ref{subsec:num} that maps each interval of $F_n$ to
 the last letter $s_{n+1}$   of the prefix  $S(\alpha)_{[1,n+1]}$ 
 of the characteristic Sturmian word $S(\alpha)$.
The labelling function  $\rho=\{\rho_n\}_{n \geq 1}$
of 
the Farey sequence of partitions thus  works as follows. 
We begin with the coding of  the unique interval in $F_0$ as the empty word. Then, there are two cases: when the mediant of an interval of $F_n$ can be added as an endpoint of $F_{n+1}$, then  this interval is subdivided and the coding of the two subintervals is made in the lexicographic order by adding a letter in $\{0,1\}$. When the mediant cannot be added, the interval is not subdivided, and the coding is extended in a unique way by adding the next letter in its palindromic completion. This is illustrated in Figure \ref{fig:s3}.
 \begin{figure}[ht]
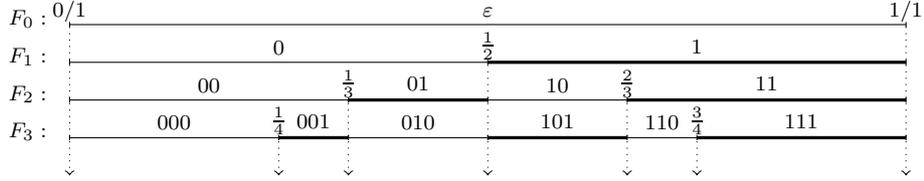

\palindomic{}
\caption{The labelling of the  Farey partition  $F_n$  for $n\le 3$.\label{fig:s3}}
\end{figure} 

\begin{rmk} \label{rmk:farey}Consider the case where we consider two identical sequences of partitions  ${\mathcal P}={\mathcal P}^1={\mathcal P}^2$  in Definition \ref{def:Lochs}, with ${\mathcal P}$  log-balanced (a.e.\ or in measure).
One can have $ L_n(x, {\mathcal P}, {\mathcal P})  > n$
if $I_n (x) =I_{n+1} (x)$. This is  for instance the case for  the  Farey partition. The construction of the coding using the  palindromic completion which shows that 
one  might have more digits  in an  interval 
of $F_n$   than the  $n+1$  letters  of   a prefix. For instance, $(1/(n+1),1/n)\in F_n$ and $L_n(x,\mathcal F,\mathcal F)=2n$ for every $x\in(1/(n+1),1/n)$.
\end{rmk}

\subsection{The three-distance  sequence of partitions}\label{subsec:3D}
The {\em three-distance sequence of partitions} $\threeda{\alpha}$ is   defined in terms of the corresponding endpoints of each partition $\threedna{\alpha}{n}$. 
The set $E_n^{\threeda{\alpha}}$ of endpoints of $\threedna{\alpha}{n}$ is given by $\{1\}\cup \{ (i  \alpha)\bmod 1 :\, 0 \leq  i \leq n\}$. We remark that, when $\alpha$ is irrational, these determine exactly $n+1$ intervals.
\subsubsection{Almost everywhere log-balancedness} 

It turns out, somewhat surprisingly, see \cite{Sos:1958,Suranyi:1958,Slater:1964} or the survey \cite{AlBe}, that the lengths of the intervals in $\threedna{\alpha}{n}$ can only take at most three different values. This classical result is known as the {\em three-distance theorem}, which we cite below. We make use of this theorem to derive the weight function of the three-distance partition for almost every $\alpha$. See \cite{CV:16} for a probabilistic study  of the lengths in the case where there are two distances,  \cite{PSZ:16} for a study of the distribution of the lengths when averaging over $\alpha$, and see  also  \cite{ZA:99}.
\begin{theorem}[The three-distance theorem]
\label{thm:3d}
Let $\alpha\in(0,1)$ an irrational number and $n$ be a positive integer. The
points $\{1\}\cup \{ (i  \alpha)\bmod 1:\, 0 \leq  i \leq n\}$
partition the unit interval $[0,1]$
into $n+1$ intervals, the lengths of which take at most three values,
one being the sum of the other two.

More precisely, let $\{p_k/q_k\}_{k\in\mathbb N}$ and $\{a_k\}_{k\in\mathbb N}$ be the
sequences of the convergents and partial quotients associated with $\alpha$
in its continued fraction expansion.
There exist unique integers $k$, $m$, and $r$ such that
\[ n=mq_k+q_{k-1}+r,\ k\geq 0,\ 1\leq m \leq a_{k+1},\text{ and }0\leq r < q_k. \]
Let $\eta_k= (-1)^k(q_k \alpha-p_k)
$ for every $k\in\mathbb N$. Then, the unit interval is divided by the points $\{\alpha\},\{2\alpha\}, \ldots,\{n\alpha\}$ into $n+1$ intervals which satisfy that: \begin{itemize}
\item
$n+1-q_k$ of them have length $\eta_k$ (which is the smallest of the three
lengths), \item
$r+1$ of them have length $\eta_{k-1}-m\eta_k$, and
\item
$q_k-(r+1)$ of them have
length $\eta_{k-1}-(m-1)\eta_k$ (which is the largest of the three lengths).
\end{itemize}
\end{theorem}

There is an interesting  connection  between   this theorem and 
frequencies of  words that occur in a   characteristic Sturmian word (see Section \ref{subsec:sturm}). Indeed,
the  lengths  of intervals in Theorem \ref{thm:3d} coincide with the frequencies of factors  of length $n+1$ \cite{VB:96}.  We will repeatedly use the following result.

\begin{theorem}[{\cite[Theorems~9 and 13]{khinchin}}]\label{thm:Khinchin9and13} Let $\alpha\in(0,1)$ an irrational and $\{p_k/q_k\}_{k\in\mathbb N}$ its sequence of convergents. If $k\in\mathbb N$ and $\eta_k=(-1)^k(q_k\alpha-p_k)$, then
\[\frac{1}{q_{k+1}+q_k}<\eta_k <\frac{1}{q_{k+1}}. \]
\end{theorem}

\subsubsection{Weight function for most three-distance sequences of partitions $\threeda{\alpha}$ }
We will prove that, for almost every  $\alpha$, the three-distance   sequence of partitions $\threeda{\alpha}$ is a.e.\ log-balanced as a consequence of the following lemma.

\begin{lemma}\label{lem:3Dalpha-loglogn}
For almost every $\alpha$ in $(0,1)$, the three-distance sequence of partitions $\threeda{\alpha}$ satisfies
\[
-\log |I^{\threeda\alpha}_n(x)| = \log n + O(\log \log n)\text{ as }n\to\infty\,,\quad\text{for every }x\not \in E^{\threeda{\alpha}},
\]
where the hidden constant in the $O$-term might depend on $\alpha$, but can be chosen so that the result holds for every $x\not \in E^{\threeda{\alpha}}$. Moreover, the limit $-\log|I_n^{\threeda{\alpha}}(\alpha)|/\log n\to 1$ as $n \to \infty$ holds for any irrational $\alpha\in(0,1)$ so that $\log a_{k}(\alpha)=o(k)$ as $k\to \infty$.
\end{lemma}
\begin{proof} Let $\alpha$ be an irrational in $(0,1)$. We keep the notation introduced in Theorem~\ref{thm:3d}. In particular, $\eta_k:=| q_k(\alpha)\cdot\alpha-p_k(\alpha)|$ for each $k\in\mathbb N$.

Let $n$ be a positive integer and choose $k\in\mathbb N$ so that
\[
    q_k+q_{k-1}\leq n < q_{k+1}+q_k.
\]
(we recall $q_0=1$ and  $q_{-1}=0$).
Let $x\in[0,1]\setminus E^{\threeda{\alpha}}$. By Theorem~\ref{thm:3d},
\[ \eta_{k}\leq |I^{\threeda\alpha}_n(x)|\leq \eta_{k-1}. \] 
By Theorem~\ref{thm:Khinchin9and13}, $\log q_{k+1}<-\log \eta_k<\log (q_{k+1} + q_k)$ and we obtain
\[  \log q_k<-\log\eta_{k-1}\leq-\log\vert I^{\threeda{\alpha}}_n(x)\vert\leq-\log\eta_k<\log(q_{k+1}+q_k). \]
Since $q_{k+1}+q_k = (a_{k+1}+1)q_k + q_{k-1}
<(a_{k+1}+2)q_k\leq (a_{k+1}+2)n$, 
\begin{align*}
    -\log\vert I^{\threeda{\alpha}}_n(x)\vert
     \leq\log(q_{k+1}+q_k)<\log n+\log(a_{k+1}+2).
\end{align*}
But we also obtain $q_k>(q_{k+1}+q_k)/(a_{k+1}+2)>n/(a_{k+1}+2)$, which implies
\begin{align*}
    -\log\vert I^{\threeda{\alpha}}_n(x)\vert
    \geq\log q_k
    >\log n-\log (a_{k+1}+2).
\end{align*}
The Borel--Bernstein Theorem (Theorem \ref{theo:BB}) implies that, for a.e.\ $\alpha$, $a_{k+1}\leq k^2$ for $k$ large enough. Moreover, $k=O(\log n)$ as $n\to\infty$ for every irrational $\alpha\in(0,1)$, as $2^{(k-1)/2}\leq q_k(\alpha)\leq n$.
Therefore, %
$\log(a_{k+1}+2)=O(\log\log n)$ as $n\to\infty$, for a.e.\ $\alpha$. This proves the first assertion of the lemma.

The second assertion of the lemma follows from the fact that if $\log(a_{k})=o(k)$ as $k\to\infty$, then $\log(a_{k+1}+2)= o(\log n)$ as $n\to\infty$.\end{proof}

We have the following consequence of Lemma~\ref{lem:3Dalpha-loglogn}.

\begin{proposition}\label{prop:weightthreeda} For almost every $\alpha$ in $(0,1)$, the three-distance  sequence of partitions $\threeda{\alpha}$ is a.e.\ log-balanced with respect to the Lebesgue measure with weight function
\[ f_{\threeda{\alpha}}(n)=\log n\quad\text{for each }n\geq 1. \]
More specifically, it holds for any irrational $\alpha\in(0,1)$ whose continued fraction expansion satisfies $\log a_k(\alpha)=o(k)$ as $k\to\infty$.
\end{proposition}

\begin{rmk}
The above result holds, for instance, for $\alpha=\phi^{-2}$, where $\phi$ denotes the golden ratio, and for $\alpha=e-2= [0;  \overline{1,2k,1}_{k=1}^{\infty}]$ (see e.g., \cite{Cohn}). %
\end{rmk}

\subsubsection{Sequences of partitions which are not log-balanced}\label{subsubsec:3D}

Next proposition will provide  an uncountable family of numbers $\alpha$ for which the corresponding sequence of partitions $\threeda{\alpha}$ is not log-balanced. 

\begin{proposition}\label{prop:nobalanced} Fix any real number $s>0$. We define 
$\alpha=\alpha(s)\in (0,1)$ as the number whose continued fraction expansion satisfies the following recurrence relation
\[a_1(\alpha)=1\quad\text{and}\quad a_{k+1}(\alpha)=\lceil q_k^s(\alpha)\rceil\quad\text{for each }k\ge 1.  \]
Then, the corresponding $\threeda{\alpha}$
sequence of partitions is not log-balanced in measure (thus, neither %
a.e.\ log-balanced) with respect to the Lebesgue measure. 

\end{proposition}

\begin{proof}
We will prove that
\begin{align*}
\frac{-\log|I^{\threeda{\alpha}}_{n}(x)|}{\log n}\end{align*}
does not converge in measure with respect to the Lebesgue measure as $n\to\infty$. For this purpose, let us define $\{{n}_k\}_{k\in \mathbb N}$ as follows: for each $k\in\mathbb N$,
\begin{equation*}{n}_k=m_kq_k+q_{k-1}+ r_k, \quad \text{with}\quad m_k=\left\lceil \frac{a_{k+1}}{2}\right\rceil\quad \text{and}\quad r_k=q_k-1.
\end{equation*}
According to the three-distance theorem, there are only two different lengths among the intervals in $\threedna{\alpha}{n_k}$. More precisely, these $n_k+1$ intervals can be partitioned into two types:
\begin{itemize}
    \item $n_k+1-q_{k}$ intervals of length $\eta_k$ and
    \item $r_k+1=q_k$ intervals of length $\Delta_k:=\eta_{k-1}-m_k\eta_k$.
\end{itemize}
On the basis of this remark, we will prove that for our choice of $\{n_k\}_{k\in \mathbb N}$ and  for each $x\in [0,1]\setminus E^{\threeda{\alpha}}$, one has:
\begin{equation}\label{eq:g(x,k)+O(1/k)}
    \frac{-\log|I^{\threeda{\alpha}}_{n_k}(x)|}{\log n_k}=g(x,k)
    +O\left(\frac1{k}\right)\quad\text{as }k\to\infty,
\end{equation}
where
\begin{align*}
    g(x,k):=\begin{cases}
      1&\text{if } |I^{\threeda{\alpha}}_{n_k}(x)|=\eta_k,\smallskip\\
      \dfrac{1}{s+1}& \text{if }  |I^{\threeda{\alpha}}_{n_k}(x)|=\Delta_k,
    \end{cases}
\end{align*}
and the constant hidden in the $O$-term does not depend on $x$.

In order to estimate the three quantities
$n_k$, $\eta_k$, and $\Delta_k$ by a convenient power of $q_k$, we will bound $m_k$ as follows:
\begin{equation}\label{Eq:boundmk}\frac{1}{2}q_k^s\le \frac{a_{k+1}}{2}\le m_k<\frac{a_{k+1}}{2}+1< \frac{1}{2}(q_k^{s}+1)+1=\frac{1}{2}q_k^{s}\left(1+\frac{3}{q_k^s}\right).
\end{equation}
Hence, each depth $n_k$ satisfies
\begin{equation}\label{Eq:depth}\frac{1}{2}q_k^{s+1}\le n_k<m_kq_k+2q_k= %
\frac{1}{2}q_k^{s+1}\left(1+\frac{7}{q_k^{s}}\right).\end{equation}
Now, we bound the distance $\eta_k$. Since $q_k^s\leq a_{k+1}<q_k^s+1$, the continuant $q_{k+1}$ satisfies
\begin{align*}
    q_{k+1}&< (q_{k}^s+1)q_k+q_{k-1}<(q_{k}^s+1)q_k+q_k=q_k^{s+1}\left(1+\frac{2}{q_k^s}\right)\quad\text{and} \\
    q_{k+1}&> q_k^{s+1}. 
\end{align*}
Thus, by Theorem~\ref{thm:Khinchin9and13}, each length $\eta_k$ satisfies
\begin{equation}\label{Eq:distance1}\frac{q_k^{-(s+1)}}{1+{3}/{q_k^s}}\le 
\frac{1}{q_{k+1}+q_k}<\eta_k <\frac{1}{q_{k+1}}\le {q_k^{-(s+1)}}.
\end{equation}
Inequalities \eqref{Eq:depth} and \eqref{Eq:distance1} prove that $\log n_k=(s+1)\log q_k+ O(1)$ and also $-\log \eta_k=(s+1)\log q_k+ O(1)$ as $k\to\infty$. Since $q_k=q_k(\alpha)$ has, at least,  exponential growth with respect to $k$, then
\begin{equation}\label{eq:eta-n}
    \frac{-\log\eta_k}{\log n_k}
    =1+O\left(\frac 1k\right)\quad\text{as }k\to\infty.
\end{equation}

In order to bound the length $\Delta_k$, we remark that the above upper bound on $\eta_k$ and the upper bound on $m_k$ given in  \eqref{Eq:boundmk} yield
\begin{equation}\label{eq:m_keta_k-upper}
    m_k\eta_k< \frac{q_k^{-1}}{2}\left(1+\frac{3}{q_k^{s}}\right).
\end{equation}
By construction, the sequence of partial quotients $\{a_k\}$ tends to infinity as $k$ does. Hence, it is possible to choose a large enough $k_0\in\mathbb N$  so that  $q_{k-1}/q_k\le 1/3$. Without loss of generality, we assume $k_0$ is large enough so that $3/q_k^s\le 1/4$  for any $k\ge k_0$. For such a $k$, it follows by Theorem~\ref{thm:Khinchin9and13} that
\begin{align*}
\Delta_k&=\eta_{k-1}-m_k\eta_k<q_k^{-1}\quad\text{and}\\
\Delta_k&\ge q_k^{-1}\left(\frac{1}{1+\dfrac{q_{k-1}}{q_k}}\right)-\frac{q_k^{-1}}{2}\left(1+\frac{3}{q_k^{s}}\right)\ge %
\frac 18q_k^{-1}.
\end{align*}
These inequalities prove that $-\log\Delta_k=\log q_k+O(1)$ but, according to \eqref{Eq:depth}, $-\log n_k=(s+1)\log q_k+O(1)$, as $k\to\infty$. Thus, as the continuants $q_k$ grow at least exponentially with $k$, this yields
\begin{equation}\label{eq:Delta-n}
    \frac{-\log\Delta_k}{\log n_k}
    =\frac 1{s+1}+O\left(\frac 1k\right)\quad\text{as }k\to\infty.
\end{equation}

Equation \eqref{eq:g(x,k)+O(1/k)} now follows from \eqref{eq:eta-n} and~\eqref{eq:Delta-n}. Notice that \eqref{eq:g(x,k)+O(1/k)} implies that there are three possibilities for the sequence  ${-\log|I_{n_k}^{\threeda{\alpha}}(x)|}/{\log n_k}$ depending on $x\in[0,1]\setminus E^{\threeda{\alpha}}$: (i) it has a limit equal to 1; (ii) it has a limit equal to $1/(s+1)$; or (iii) it has no limit. 
To end the proof of this proposition, we now show that  there exists $\epsilon>0$ so that  the measures of the sets
\[F_{k,\epsilon}=\left\{x\in [0,1]\setminus E^{\threeda{\alpha}}:\, \left|\frac{-\log|I_{n_k}^{\threeda{\alpha}}(x)|}{\log n_k}-1\right|< \varepsilon\right\} \]
 belong to $\left[\frac 14,\frac 34\right]$
for $k$ large enough. Consider any  $\epsilon$ with $0<\epsilon<{1}/{(2(s+1))}$. Because of \eqref{eq:g(x,k)+O(1/k)},  there is a $k_1\in \mathbb N$ so that,  for each $k\ge k_1$, the set $F_{k,\epsilon}$ coincides with the set $G_k$ given by 
\[G_{k}=\left\{x\in [0,1]\setminus E^{\threeda{\alpha}}:\,|I^{\threeda{\alpha}}_{n_k}(x)|=\eta_k\right\}.\]
The number of intervals of length $\eta_k$ is $n_k+1-q_k$. Thus, one  has
\[|G_k|=(n_k+1-q_k)\eta_k=(m_kq_k+q_{k-1})\eta_k.\]
The bounds on $m_k\eta_k $ and $\eta_k$ in \eqref{eq:m_keta_k-upper} and \eqref{Eq:distance1} yield, for $k$ large enough so that $5/q_k^{s}\le 1/2$,
\begin{align*}
    |G_k|&\le (m_k+1)\eta_kq_k\le \frac{1}{2}q_k^{-1}\left(1+\frac 5{q_k^{s}}\right)q_k\le \frac 34\quad\text{and}\\
    |G_k|&\ge m_kq_k\eta_k\ge \frac{1}{2}q_k^{s+1}\left(\frac{q_k^{-s-1}}{1+3/q_k^s}\right)\ge\frac 14. 
\end{align*}
Finally, we have proved that, for each $\epsilon$ with $0<\varepsilon<{1}/{(2(s+1))}$ and sufficiently large $k$, the following holds $\left|F_{k,\epsilon}\right|=|G_k|\in [1/4,3/4]$. Then, the sequence $-\log|I_{n_k}^{\threeda{\alpha}}(x)|/{\log n_k}$ cannot converge in measure as $k\to\infty$, and thus the proof of the proposition is complete. \end{proof}

\section{On the Farey and   continued fraction sequences of partitions: an explicit case}\label{sec:FCF}

In this section, we look at the relationship between the Farey and  the continued fraction sequences of partitions. Both cases are of particular interest as the Lochs indexes in both conversion directions can be computed explicitly, as  demonstrated with Proposition \ref{le:continued fractions Farey} and \ref{prop:indexFtoCF}, which moreover provide   direct proofs of the Lochs-type theorems for these cases. We also show stronger results, when compared to our general theorems,  including a probabilistic version of  the error term in the case of the Lochs index from continued fraction to Farey.

More precisely, for the probabilistic error term when going from continued fractions to Farey, namely Theorem~\ref{teo:prob-term-cf-farey}, we apply a strong result regarding the
asymptotic distribution of the logarithm of the continuants $\log q_n$ as $n\to\infty$. This result is a refinement over the classical Khinchin--L\'evy's Theorem (see~\eqref{eq:Levy}), which tells us that $(1/n)\log q_n\to\pi^2/(12 \log 2)$ as $n\to\infty$. The refinement comes in the form of an extra term, of probabilistic nature: the random variable $\log q_n$ is asymptotically Gaussian (see, e.g., \cite[Thm.~1]{FV:98} and \cite{Ibragimov}), as recalled below.

\begin{theorem}\label{thm:Gaussian}
Let $h=\pi^2/(6\log 2)$ be the entropy of the continued fraction sequence of partitions. Then, there exists a positive constant $B$ so that the sequence
\begin{equation}\label{eq:logqk-gaussian}
Z_n(x)=\frac{\log q_n(x)-\frac h2n}{B\sqrt{n}}
\end{equation}
is asymptotically Gaussian; i.e., for every $t\in\mathbb R$,
\[\left|\{x\in [0,1]:\,Z_n(x)<t
\}\right|= \frac{1}{\sqrt{2\pi}}\int_{0}^t e^{-w^2/2}dw+O\Big(\frac{1}{\sqrt n}\Big)
\]
as $n\to \infty$. (We recall that $\vert\cdot\vert$ stands for the Lebesgue measure.)
\end{theorem}

\subsection{From    continued fractions   to   the Farey  sequence of partitions}\label{sssec:CF->F}

We begin by describing the Lochs index for the conversion  from the  continued fraction sequence of partitions  to  the Farey one.

\begin{lemma}\label{lem:ICFIF} For each irrational $x\in(0,1)$ and each $n\in\mathbb N$, one has
\begin{equation}\label{eq:CF=F-depths}
    I_n^{\mathcal{CF}}(x)=I_k^{\mathcal F}(x),\quad\text{where }k=q_n+q_{n-1}-1.
\end{equation}
\end{lemma}
\begin{proof} On the one hand, the interval $I_n^{\mathcal{CF}}(x)$
of the continued fraction sequence of partitions  is simply, up to reordering the endpoints depending on the parity of $n$, \begin{equation}\label{eq:InCF}
I_n^{\mathcal{CF}}(x) = \left(\frac{p_n}{q_n},\frac{p_n+p_{n-1}}{q_n+q_{n-1}}\right)\,.
\end{equation}

On the one hand, by Proposition \ref{prop:FareyFF}, the Farey interval of depth $k$ containing $x$, up to reordering of the endpoints depending on the parity of $m$, is given by
\begin{equation}\label{eq:IkF}
 I_k^{\mathcal{F}}(x) =  \Big(\frac{p_m}{q_m},\frac{(r+1) p_m+p_{m-1}}{(r+1) q_m+q_{m-1}}\Big)\,,
\end{equation}
where $m$ and $r$ are the only integers such that
\[ (r+1) q_m+q_{m-1}\leq k+1<(r+2)q_m+q_{m-1},\ m\geq 0,\text{ and }0 \leq r < a_{m+1}.
\]
Notice that if $k=q_n+q_{n-1}$, then $r=0$ and $m=n$, which shows that $I_k^{\mathcal F}(x)=I_n^{\mathcal{CF}}(x)$.
\end{proof}

 \begin{proposition} \label{le:continued fractions Farey}
For any irrational number $x\in (0,1)$ and any integer $n\ge 1$, one has  
$$L_n(x,\mathcal{CF},\mathcal{F})=2q_n(x)+q_{n-1}(x)-2.$$
\end{proposition}
\begin{proof} By Lemma~\ref{lem:ICFIF},    $I_n^{\mathcal{CF}}(x)=I_k^{\mathcal F}(x)$,  where $k=q_n+q_{n-1}-1$. This interval will split into two subintervals in the Farey sequence of partitions for the first time at depth $2q_n+q_{n-1}-1$ (for this is the denominator of the mediant of its endpoints). Hence, the proposition follows.\end{proof}

Theorem~\ref{thm:Gaussian}, together with the above result, yields the following theorem which gives a Lochs-type theorem for the conversion from the continued fraction sequence of partitions to the Farey one with an error term satisfying an asymptotically Gaussian law.

\begin{theorem}
\label{teo:prob-term-cf-farey} If $Z_n$ and $B$ are as in Theorem~\ref{thm:Gaussian} and $x$ is an irrational in $(0,1)$, then
\begin{equation}
    \frac{\log L_n(x, \mathcal{CF},\mathcal{F})}{n} = \frac{\pi^2}{12 \log 2} + \frac{B}{\sqrt{n}} Z_n(x)+ O\left(\frac 1n\right),\quad\text{as }n\to\infty,
\end{equation}
with $Z_n$ asymptotically Gaussian. In particular,
\[ \lim_{n\to\infty}\frac{\log L_n(x, \mathcal{CF},\mathcal{F})}{n} = \frac{\pi^2}{12 \log 2}\quad\text{a.e.}, \]
with respect to the Lebesgue measure.
\end{theorem}
\begin{proof}
Let $x\in(0,1)$ an irrational. Observe that, by Proposition~\ref{le:continued fractions Farey} and each positive integer $n$,
\[ \log L_n(x,\mathcal{CF},\mathcal F)=\log (2 q_n+q_{n-1}-2) = \log q_n + \log\left(2+\tfrac{q_{n-1}-2}{q_n}\right)
. \]
As the term $\log\left(2+\tfrac{q_{n-1}-2}{q_n}\right) $ is upper bounded by $\log 3$, then 
\begin{align*}
  \frac{\log L_n(x,\mathcal{CF},\mathcal F)}{n}
  =\frac{\log q_n}n+O\left(\frac 1n\right)
  =\frac{\pi^2}{12\log2}+\frac B{\sqrt n}Z_n(x)+O\left(\frac 1n\right)
\end{align*}
as $n\to\infty$ by Theorem~\ref{thm:Gaussian}.
\end{proof}

\subsection{Coming back:  from Farey to continued fractions}

In the case of the  conversion from  the Farey sequence of partitions to  the continued fraction one, the Lochs-type result we obtain comes directly from the expression of the  Lochs index  from  Proposition~\ref{prop:fare-to-cf} below. This is an almost everywhere result that is not covered by our more generally applicable Theorem~\ref{thm:resultae}.  Indeed, 
the difficulty in applying Theorem~\ref{thm:resultae} in this case stems from the fact that assertion (\ref{it:delta}) does not hold; this is because the weight function $f_1$ of the Farey  sequence of partitions is slowly-increasing. Nevertheless, our general result in measure (Theorem~\ref{thm:resultinmeasure}) does apply in this case, which is enough to ensure that the limit \eqref{eq:fare-to-cf} holds but only in measure.

\begin{proposition} \label{prop:indexFtoCF}
Let $x\in(0,1)$ be an irrational number and let $n$ be a positive integer. If $k=k(x,n)$ %
is the only positive integer such that
\begin{equation}\label{eq:defk}
q_k(x)+q_{k-1}(x)\leq n + 1 < q_{k+1}(x)+q_k(x)\,,
\end{equation}
then
\[ L_n(x,\mathcal{F},\mathcal{CF})=k. \]
\end{proposition}
\begin{proof} By Lemma~\ref{lem:ICFIF}, $I_k^{\mathcal{CF}}(x)=I_{n'}^{\mathcal F}(x)$ and $I_{k+1}^{\mathcal{CF}}(x)=I_{n''}^{\mathcal F}(x)$ where $n'=q_k+q_{k-1}-1$ and $n''=q_{k+1}+q_k-1$. By assumption, $n'\leq n<n''$. Thus,
\[ I_{k+1}^{\mathcal{CF}}(x)=I_{n''}^{\mathcal F}(x)\subseteq I_n^{\mathcal F}(x)\subseteq I_{n'}^{\mathcal F}(x)=I_k^{\mathcal{CF}}(x). \]
Notice that one of the endpoints of $I_{k+1}^{\mathcal{CF}}(x)$ is $\frac{p_{k+1}+p_k}{q_{k+1}+q_k}$, which does not belong to $F_n$ (because $q_{k+1}+q_k>n+1)$. Hence, $I_{k+1}^{\mathcal{CF}}(x)\subsetneq I_n^{\mathcal F}(x)$ and, in particular, $I_n^{\mathcal F}(x)\nsubseteq I_{k+1}^{\mathcal{CF}}(x)$. We conclude that $k$ is the largest integer $k'$ such that $I_n^{\mathcal{F}}(x)\subseteq I_{k'}^{\mathcal{CF}}(x)$, i.e., $k=L_n(x,\mathcal F,\mathcal{CF})$.\qedhere
\end{proof}

\begin{proposition}%
\label{prop:fare-to-cf}
The following holds with respect to the Lebesgue measure:
\begin{equation}\label{eq:fare-to-cf}
\lim_{n\to\infty}\frac{L_n(x;\MC{F},\MC{CF})}{\log n} =\frac{12\log 2}{\pi^2}\quad\text{a.e.}
\end{equation}
\end{proposition}
\begin{proof}%
Let $x\in(0,1)$ an irrational and $n$ be a positive integer. By Proposition~\ref{prop:indexFtoCF}, if $k=k(x,n)$ is the only positive integer satisfying \eqref{eq:defk}, then $L_n(x,\mathcal F,\mathcal{CF})=k(x,n)$ and the following inequalities hold
\[ 2 q_{k(x,n)-1}(x)\leq   n  \leq 2 q_{k(x,n)+1}(x). \]
As $k(x,n)\to\infty$ as $n\to\infty$, Khinchin--L\'evy's Theorem (see \eqref{eq:Levy}) yields $\log q_{k(x,n)+1}/\log q_{k(x,n)-1}\to 1$ as $n\to\infty$ a.e. Hence, $\log n/\log q_{k(x,n)}\to 1$ as $n\to\infty$ a.e. Therefore, again by Khinchin--L\'evy's Theorem, we have that
\[ \lim_{n\to\infty}\frac{L_n(x,\mathcal F,\mathcal{CF})}{\log n}
   =\lim_{n\to\infty}\frac{k(x,n)}{\log n}
  =\lim_{n\to\infty}\frac{k(x,n)}{\log q_{k(x,n)}}=\frac{12\log 2}{\pi^2}\quad\text{a.e.}\qedhere \]
\end{proof}
 
 It would be interesting to characterize
 \[ \frac{L_n(x,\mathcal F,\mathcal{CF}) - \tfrac{\pi^2}{12\log 2} \log n}{\sqrt{\log n}} \]
 in law.
 The objective would be to obtain an analog to Theorem~\ref{teo:prob-term-cf-farey} but going from Farey to continued fractions.

\section{Lochs-type theorems for log-balanced sequences of partitions}\label{sec:convtheo}

In this section, we give the proofs of our general Lochs-type theorems: Theorems~\ref{thm:resultae} and~\ref{thm:resultinmeasure}.
The structure of this section is as follows. In Section~\ref{ssec:Lochs-type-defs}, we introduce some auxiliary notions which are useful for the proofs throughout this section. In Sections~\ref{ssec:Lochs-type theorem-ae} and~\ref{ssec:Lochs-type theorem-inm}, we prove our main result a.e.\ and in measure, respectively. %

\subsection{Preliminaries}\label{ssec:Lochs-type-defs}

We follow here the notion  of $\epsilon$-goodness introduced in \cite{DaField:01}.
\begin{definition} Let $\mathcal P=\{P_n\}_{n\in\mathbb N}$ be a sequence of partitions and let $f:\mathbb N\to\mathbb R$. If $x\in[0,1]\setminus E$, $n\in\mathbb N$, and $\epsilon>0$, we say that $I_n(x)\in P_n$ is \emph{$\epsilon$-good for $\mathcal P$ and $f$} if
\[ (1-\epsilon)f(n)<-\log\lambda(I_n(x))<(1+\epsilon)f(n) \]
or, equivalently,
\[ e^{-(1+\epsilon)f(n)}<\lambda(I_n(x))<e^{-(1-\epsilon)f(n)}. \]
\end{definition}

\begin{notation} In the proofs given in this section, for each $i=1,2$, we say that  $I_n^{i}(x)$ is 
\emph{$\epsilon$-good} if $I_n^{i}(x)$ is $\epsilon$-good for $\mathcal P^i$ and $f_i$.\end{notation}

\begin{definition} If $f:\mathbb N\to\mathbb R$ is a nondecreasing function  such that $f(n)$ tends to $+\infty$ as $n\to\infty$, we denote by $f^{[-1]}$ the function $f^{[-1]}:\mathbb R\to\mathbb N$ defined as  $f^{[-1]}(y)=\min\{n\in\mathbb N:\,f(n)\geq y\}$.
\end{definition}

\begin{rmk}\label{rmk:f[-1]} All the following assertions are immediate consequences of the definition:
\begin{enumerate}[(i)]
    \item\label{it:rmk1} $f^{[-1]}$ is nondecreasing;
    
    \item\label{it:rmk2}
    $f^{[-1]}(y)\to\infty$ as $y\to+\infty$;
    
    \item\label{it:rmk3} for each $y\in\mathbb R$,  $f(f^{[-1]}(y))\geq y$;

    \item\label{it:rmk5} if $n\in\mathbb N$ and $y\in\mathbb R$ are such that $n<f^{[-1]}(y)$, then $f(n)<y$; and
    
    \item\label{it:rmk6} if $f$ is strictly increasing and $\tilde f:[0,+\infty)\to\mathbb R$ is any strictly increasing function such that $\tilde f(n)=f(n)$ for each $n\in\mathbb N$, then $f^{[-1]}(y)=\lceil\tilde f^{-1}(y)\rceil$ for every $y\in\mathbb R$.
\end{enumerate}\end{rmk}

The following lemma tells us that, if $f$ is not  allowed  to grow too fast, we can reverse $f^{[-1]}$ by $f$ asymptotically.
\begin{lemma}
Let $f\colon \mathbb{N}\to \mathbb{R}$ be a nondecreasing function tending to infinity. If $f(n+1)-f(n)=o(f(n))$ as $n\to\infty$, then $f(f^{[-1]}(y))/y\to 1$ as $y\to+\infty$ (over $y\in \mathbb{R}$).
\label{lemma:ff-1}
\end{lemma}
\begin{proof}
By Remark~\ref{rmk:f[-1]} (\ref{it:rmk2}), $f^{[-1]}(y)\to\infty$ as $y\to+\infty$. Thus, the fact that $f(n+1)-f(n)=o(f(n))$ as $n\to\infty$ implies $f(f^{[-1]}(y))/f(f^{[-1]}(y)-1)\to 1$ as $y\to+\infty$. Moreover, by items~(\ref{it:rmk3}) and~(\ref{it:rmk5}) of Remark~\ref{rmk:f[-1]}, $f(f^{[-1]}(y)-1)<y\leq f(f^{[-1]}(y))$. Therefore, dividing through by $f(f^{[-1]}(y))$ implies the result, as both extremes tend to $1$ as $y\to+\infty$.
\end{proof}

\subsection{Lochs-type theorem a.e.}\label{ssec:Lochs-type theorem-ae}

In this section, we prove our general Lochs-type theorem a.e., namely  Theorem~\ref{thm:resultae}. This result is a consequence of Propositions~\ref{prop:limsupae} and~\ref{prop:liminfae}.

\begin{proposition}\label{prop:limsupae}
Let $\mathcal P^1$ and $\mathcal P^2$ be two a.e.\ log-balanced sequences of partitions with weight functions $f_1$ and $f_2$, respectively,  with respect to a measure $\lambda$. If $f_2$ is nondecreasing, then
\begin{align*}
    \limsup_{n\to\infty}\frac{f_2(L_n(x,\mathcal P^1,\mathcal P^2))}{f_1(n)}\leq 1\quad\text{a.e. }(\lambda).
\end{align*}
\end{proposition}
\begin{proof} %

Let $\epsilon\in(0,1)$. For each $n\in\mathbb N$, let $m(n):=f_2^{[-1]}\bigl(\frac{1+\epsilon}{1-\epsilon}f_1(n)\bigr)$. Let $m\in\mathbb N$ such that $m\geq m(n)$. Since $f_2$ is nondecreasing, $f_2(m)\geq f_2(m(n))\geq\frac{1+\epsilon}{1-\epsilon}f_1(n)$, by Remark~\ref{rmk:f[-1]} \eqref{it:rmk3}. Thus, $e^{-(1+\epsilon)f_1(n)}\geq e^{-(1-\epsilon)f_2(m)}$ for all $n\in\mathbb N$.
Hence, for each $x\in[0,1]\setminus(E^1\cup E^2)$, each $n\in\mathbb N$, and each $m\in\mathbb N$ such that $m\geq m(n)$, if $I^{1}_n(x)$ and $I^{2}_{m}(x)$ are $\epsilon$-good, then
\begin{align*}
    \lambda(I^{1}_n(x))>e^{-(1+\epsilon)f_1(n)}\geq e^{-(1-\epsilon)f_2(m)}>\lambda(I^{2}_m(x))
\end{align*}
and, in particular, $I_n^{1}(x)\nsubseteq I^{2}_m(x)$. Therefore, for each $x\in [0,1]\setminus(E^1\cup E^2)$ and each $n\in\mathbb N$ such that both $I_n^{1}(x)$ is $\epsilon$-good and $I_m^{2}(x)$ is $\epsilon$-good for each $m\geq m(n)$, we have that
$L_n(x,\mathcal P^1,\mathcal P^2)<m(n)$.

Since $-\log\lambda(I_n^{1}(x))/f_1(n)\to 1$ as $n\to\infty$ a.e., we have that $I^{1}_n(x)$ is $\epsilon$-good for $n$ large enough a.e. Notice that, since $f_1(n)\to+\infty$ as $n\to\infty$, Remark~\ref{rmk:f[-1]} \eqref{it:rmk2} implies $m(n)\to\infty$ as $n\to\infty$. Thus, the fact that $-\log\lambda(I_m^{2}(x))/f_2(m)\to 1$ as $m\to\infty$ a.e.\ implies that for a.e.\ $x$, for $n$ large enough (depending on $x$), $I_{m}^{2}(x)$ is $\epsilon$-good for each $m\in\mathbb N$ such that $m\geq m(n)$. Hence, $L_n(x,\mathcal P^1,\mathcal P^2)<m(n)$ for $n$ large enough a.e. Therefore, by Remark~\ref{rmk:f[-1]} \eqref{it:rmk5}, $f_2(L_n(x,\mathcal P^1,\mathcal P^2)) < \frac{1+\epsilon}{1-\epsilon}f_1(n)$ for $n$ large enough a.e. Taking $\epsilon\to 0$ gives the result.
\end{proof}

\begin{proposition}\label{prop:liminfae} Let $\mathcal P^1$ and $\mathcal P^2$ be two a.e.\ log-balanced sequences of partitions with weight functions $f_1$ and $f_2$, respectively, with respect to a measure $\lambda$ such that the following assertions hold:
\begin{enumerate}[(i)]
\item\label{it:liminfae2} $f_2$ is nondecreasing;
\item\label{it:liminfae3} for each $\eta\in(0,1)$, there exists $\epsilon>0$ such that
\begin{align*}
   \sum_{n=1}^\infty e^{-(1-\epsilon)f_1(n)+(1+\epsilon)f_2(m(n))}<+\infty,\quad\text{where }
    m(n):=f_2^{[-1]}\left((1-\eta)f_1(n)\right).
\end{align*}
\end{enumerate}
Then,
\begin{align*}
    \liminf_{n\to\infty}\frac{f_2(L_n(x,\mathcal P^1,\mathcal P^2))}{f_1(n)}\geq 1\quad\text{a.e.\ }(\lambda).
\end{align*}
\end{proposition}
\begin{proof} %

Let $\eta\in(0,1)$. For each $n\in\mathbb N$, let $\mathcal J_n$ be the set of ordered pairs $(I^{1}_n(x),I^{2}_{m(n)}(x))$ such that $I^{1}_n(x)\nsubseteq I^{2}_{m(n)}(x)$. Let $(I^{1}_n(x),I^{2}_{m(n)}(x))\in\mathcal J_n$. Since $I^{1}_n(x)$ and $I^{2}_{m(n)}(x)$ share the point $x$, necessarily $I^{1}_n(x)$ contains at least one of the endpoints of $I^{2}_{m(n)}(x)$. Thus, there are no three ordered pairs in $\mathcal J_n$ having the same second entry.

Suppose $\epsilon>0$ (we will choose $\epsilon$ later). Let $\mathcal J_{n,\epsilon}$ be the set of ordered pairs in $\mathcal J_n$ whose  entries are both $\epsilon$-good. Notice that, if $(I^{1}_n(x),I^{2}_{m(n)}(x))\in\mathcal J_{n,\epsilon}$, then
\begin{align*}
    \lambda(I^{1}_n(x))<
      e^{-(1-\epsilon)f_1(n)}\quad\text{and}\quad
    e^{-(1+\epsilon)f_2(m(n))}<\lambda(I^{2}_{m(n)}(x))
\end{align*}
and, consequently,
\begin{align*}
    \lambda(I^{1}_n(x))
    <e^{-(1-\epsilon)f_1(n)+(1+\epsilon)f_2(m(n))}{\lambda(I^{2}_{m(n)}(x))}.
\end{align*}
Let $D_{n,\epsilon}
    :=\{x:\,\bigl(I^{1}_n(x),I^{2}_{m(n)}(x)\bigr)\in\mathcal J_{n,\epsilon}\}=\bigcup_{(I^{1}_n(x),I^{2}_{\bar m(n)}(x))\in\mathcal J_{n,\epsilon}}(I^{1}_n(x)\cap I^{2}_{m(n)}(x))$.
    Hence, this gives
 \begin{align*}
    \lambda(D_{n,\epsilon})
    &=\sum_{(I^{1}_n(x),I^{2}_{m(n)}(x))\in\mathcal J_{n,\epsilon}}\lambda(I^{1}_n(x)\cap I^{2}_{m(n)}(x))
    \leq\sum_{(I^{1}_n(x),I^{2}_{m(n)}(x))\in\mathcal J_{n,\epsilon}}\lambda(I^{1}_n(x))\\
    &\leq e^{-(1-\epsilon)f_1(n)+(1+\epsilon)f_2(m(n))}\sum_{(I^{1}_n(x),I^{2}_{m(n)}(x))\in\mathcal J_{n,\epsilon}}\lambda(I^{2}_{m(n)}(x))\\
    &\leq 2e^{-(1-\epsilon)f_1(n)+(1+\epsilon)f_2(m(n))}.
\end{align*}
(Recall that there are no three ordered pairs in $\mathcal J_n$ having the same second entry.)

Let $\epsilon$ be as in assertion~\eqref{it:liminfae3}. Hence, by the Borel--Cantelli Lemma we have that
\begin{align*}
  \lambda(\{x:\,x\in D_{n,\epsilon}\text{ i.o.}\})=0.
\end{align*}
Thus, $x\notin D_{n,\epsilon}$ for $n$ large enough a.e. Since $f_1(n)\to+\infty$ as $n\to\infty$, Remark~\ref{rmk:f[-1]} \eqref{it:rmk2} ensures that $m(n)\to\infty$ as $n\to\infty$. Thus, since $f_1$ and $f_2$ are a.e.\ weight functions of $\mathcal P_1$ and $\mathcal P^2$, $I^{1}_n(x)$ and $I^{2}_{m(n)}(x)$ are $\epsilon$-good for $n$ large enough a.e. By the definition of $D_{n,\epsilon}$, necessarily $I^{1}_n(x)\subseteq I^{2}_{m(n)}(x)$ for $n$ large enough a.e. In particular, $L_n(x,\mathcal P^1,\mathcal P^2)\geq m(n)$ for $n$ large enough a.e. By assertion~\eqref{it:liminfae2}, the fact that $f_2(n)\to+\infty$ as $n\to\infty$, and Remark~\ref{rmk:f[-1]} \eqref{it:rmk3}, one has 
\[  f_2(L_n(x,\mathcal P^1,\mathcal P^2))\geq f_2(m(n))\geq(1-\eta)f_1(n)\quad\text{for $n$ large enough a.e.}
\]
Taking $\eta\to 0$ gives the result.\end{proof}

Now we are ready to prove our main result a.e.

\mainResultae
\begin{proof} %
By Propositions~\ref{prop:limsupae} and~\ref{prop:liminfae}, it suffices to check that, for each $\eta\in(0,1)$, there exists $\epsilon>0$ so that
\begin{equation}
   \sum_{n=1}^\infty e^{-(1-\epsilon)f_1(n)+(1+\epsilon)f_2(m(n))}<+\infty,\quad\text{where }
    m(n):=f_2^{[-1]}\left((1-\eta)f_1(n)\right). \label{eq:convergentseries}
\end{equation}

Let $\eta>0$. We show that $\epsilon=\eta/2$ satisfies the condition. By assertion~\eqref{it:f2equivalentn1}, we can apply Lemma~\ref{lemma:ff-1} to $f=f_2$ and $y=(1-\eta)f_1(n)$ in order to obtain $f_2(m(n))/((1-\eta)f_1(n))\to 1$ as $n\to\infty$.

Hence, since $\epsilon=\eta/2$,  for $n$ large enough one has 
$$
  -(1-\epsilon)f_1(n)+(1+\epsilon)f_2(m(n))
    =\left(-\frac{\eta^2}2+o(1)\right)f_1(n)
    \leq-\frac{\eta^2}4f_1(n).
$$
It follows from assertion~\eqref{it:delta} that \eqref{eq:convergentseries} holds. This proves the theorem.\end{proof}

\begin{rmk}\label{rem:PP}
The assumptions of the above result are sufficient conditions. %
These assumptions are however not always necessary.  The next result shows that, for instance, the conclusion of Theorem~\ref{thm:resultae} still holds in the case where $\mathcal P^1=\mathcal P^2=\mathcal F$ even if the weight function $f_{\mathcal F}(n)=2\log n$ does not satisfy assumption~(\ref{it:delta}) of the theorem.
\end{rmk}

\begin{proposition}\label{prop:xPP}
Let $\mathcal P$ be an a.e.\ log-balanced sequence of partitions with weight function $f$ with respect to a measure $\lambda$. If $f$ is nondecreasing, then
\begin{align*}
    \frac{f(L_n(x,\mathcal P,\mathcal P))}{f(n)}\quad\text{as }n\to\infty,\quad \text{a.e.}\ (\lambda).
\end{align*}
\end{proposition}
\begin{proof} Notice that, by definition of Lochs' index, $L_n(x,\mathcal P,\mathcal P)\geq n$ for each $n\in\mathbb N$. Moreover, since $f$ is nondecreasing, $f(L_n(x,\mathcal P,\mathcal P))\geq f(n)$ for every $n\in\mathbb N$. Thus, $\liminf_{n\to\infty} f(L_n(x,\mathcal P,\mathcal P)) / f(n) \geq 1$. Therefore, by Proposition~\ref{prop:limsupae}, the proof is complete.
\end{proof}

\begin{rmk} It is unclear whether assumption \eqref{it:liminfae3} in Proposition~\ref{prop:liminfae} is  necessary.  Consider in particular the case where $\mathcal P^2$ is a slight perturbation of $\mathcal P^1$ so that $f_1=f_2$. A natural question is whether it is possible, in this case, to have $\liminf_{n\to\infty}\frac{f_2(L_n(x,\mathcal P^1,\mathcal P^2))}{f_1(n)}<1$ a.e.\ or that this limit does not even exist. If possible, this situation could be regarded as revealing an intrinsic difficulty, in the sense that considering only the lengths of the intervals does not provide enough information for proving Lochs' conversion results.\end{rmk}

\subsection{Lochs-type theorem in measure}\label{ssec:Lochs-type theorem-inm}

In this section, we prove our general Lochs-type theorem in measure, namely  Theorem~\ref{thm:resultinmeasure}. This result is a consequence of Propositions~\ref{prop:limsupiinm} and~\ref{prop:liminm}.

\begin{proposition}\label{prop:limsupiinm}
Let $\mathcal P^1$ and $\mathcal P^2$ be two sequences of partitions that are log-balanced in measure with weight functions $f_1$ and $f_2$, respectively, with respect to $\lambda$, such that the following assertions hold:
\begin{enumerate}[(i)]
\item $\mathcal P^2$ is self-refining;
    
\item $f_2$ is nondecreasing.
\end{enumerate}    
Then,
\[ \limsup_{n\to\infty}\frac{f_2(L_n(x,\mathcal P^1,\mathcal P^2))}{f_1(n)}\leq 1\quad\text{in measure }(\lambda). \]
\end{proposition}
\begin{proof} %

Let $\epsilon\in(0,1)$ and let $m(n):=f_2^{[-1]}\bigl(\frac{1+\epsilon}{1-\epsilon}f_1(n)\bigr)$. Arguing as in the proof of Proposition~\ref{prop:limsupae}, $e^{-(1+\epsilon)f_1(n)}\geq e^{-(1-\epsilon)f_2(m(n))}$ for all $n\in\mathbb N$. Hence, for each $x\in[0,1]\setminus(E^1\cup E^2)$ and each $n\in\mathbb N$, if $I^{1}_n(x)$ and $I^{2}_{m(n)}(x)$ are $\epsilon$-good, then
\begin{align*}
    \lambda(I^{1}_n(x))>e^{-(1+\epsilon)f_1(n)}\geq e^{-(1-\epsilon)f_2(m(n))}>\lambda(I^{2}_{m(n)}(x))
\end{align*}
and, in particular, $I_n^{1}(x)\nsubseteq I^{2}_{m(n)}(x)$. Moreover, as we are assuming that $\mathcal P^2$ is self-refining, it follows that $I_n^{1}(x)\nsubseteq I_m^{2}(x)$ for each $m\geq m(n)$. Therefore, $L_n(x,\mathcal P^1,\mathcal P^2)<m(n)$.
Hence,
\begin{align*}
    C_n
    &:=\{x:\,L_n(x,\mathcal P^1,\mathcal P^2)\geq m(n)\}\\
    &\subseteq\{x:\,I_n^{1}(x)\text{ is not $\epsilon$-good}\}\cup\{x:\,I^{2}_{m(n)}(x)\text{ is not $\epsilon$-good}\}.
\end{align*}
By Remark~\ref{rmk:f[-1]} \eqref{it:rmk2}, $m(n)\to\infty$ as $n\to\infty$. Thus, the  fact that $f_1$ and $f_2$ are weight functions in measure implies that, for $n$ large enough,
\[ \lambda(\{x:I^{1}_n(x)\text{ is $\epsilon$-good}\})>1-\epsilon\quad\text{and}\quad
\lambda(\{x:I^{2}_{m(n)}(x)\text{ is $\epsilon$-good}\})>1-\epsilon. \]
Hence,
$\lambda(C_n)\leq 2\epsilon$ for $n$ large enough. Moreover, by Remark~\ref{rmk:f[-1]} \eqref{it:rmk5}, if $x\in[0,1]\setminus C_n$, then
\[ f_2(L_n(x,\mathcal P^1,\mathcal P^2))<\frac{1+\epsilon}{1-\epsilon}f_1(n)=\left(1+\frac{2\epsilon}{1-\epsilon}\right)f_1(n). \]
Taking $\epsilon\to 0$ gives the result.\qedhere
\end{proof}

\begin{proposition}\label{prop:liminm} Let $\mathcal P^1$ and $\mathcal P^2$ be two sequences of partitions that are log-balanced in measure with weight functions $f_1$ and $f_2$, respectively, with respect to a measure $\lambda$, such that the following assertions hold:
\begin{enumerate}[(i)]

\item\label{it:teoinm2} $f_2$ is nondecreasing;

\item\label{it:teoinm3} for each $\eta\in(0,1)$, there exists some $\epsilon>0$ such that
\[ \lim_{n\to\infty}\,(1-\epsilon)f_1(n)-(1+\epsilon)f_2(m(n))=+\infty, \text{where }m(n):= f_2^{[-1]}\left((1-\eta)f_1(n)\right). \]
\end{enumerate}
Then,
\begin{align*}
    \liminf_{n\to\infty}\frac{f_2(L_n(x,\mathcal P^1,\mathcal P^2))}{f_1(n)}\geq 1\quad\text{in measure }(\lambda).
\end{align*}
\end{proposition}
\begin{proof} %

Let $\eta\in(0,1)$ and let $\epsilon$ be as in \eqref{it:teoinm3}. Let $D_{n,\epsilon}$ be defined as in the proof of Theorem~\ref{prop:liminfae}. As argued in that proof, $\lambda(D_{n,\epsilon})\leq 2e^{-(1-\epsilon)f_1(n)+(1+\epsilon)f_2(m(n))}$. Because of the choice of $\epsilon$, $\lambda(D_{n,\epsilon})<\eta$ for $n$ large enough. Notice that, since $f_1(n)\to+\infty$ as $n\to\infty$, Remark~\ref{rmk:f[-1]} \eqref{it:rmk2} ensures that $m(n)\to\infty$ as $n\to\infty$. As $f_1$ and $f_2$ are weight functions in measure for $\mathcal P^1$ and $\mathcal P^2$, then, %
for $n$ large enough,
\[ \lambda(\{x:\,I^{1}_n(x)\text{ is $\epsilon$-good}\})>1-\eta\quad\text{and}\quad
   \lambda(\{x:\,I^{2}_{m(n)}(x)\text{ is $\epsilon$-good}\})>1-\eta. \]
Hence, since
\begin{align*}
    B_n
    &:=\{x:\,L_n(x,\mathcal P^1,\mathcal P^2)<m(n)\}\\
    &\subseteq\{x:\,I^{1}_n(x)\nsubseteq I^{2}_{m(n)}(x)\}\\
    &\subseteq D_{n,\epsilon}\cup\{x:\,I^{1}_n(x)\text{ is not $\epsilon$-good}\}\cup\{x:\,I^{2}_{m(n)}(x)\text{ is not $\epsilon$-good}\},
\end{align*}
we have that $\lambda(B_n)<3\eta$ for $n$ large enough. Moreover, by assertion~\eqref{it:teoinm2}, the fact that $f_2(n)\to+\infty$ as $n\to\infty$, and Remark~\ref{rmk:f[-1]} \eqref{it:rmk3}, if $x\in[0,1]\setminus B_n$, then
\[  f_2(L_n(x,\mathcal P^1,\mathcal P^2))\geq f_2(m(n))\geq(1-\eta)f_1(n)
\]
for $n$ large enough. Taking $\eta\to 0$ gives the result.\end{proof}

\mainResultim
\begin{proof} Arguing as in the proof of Theorem~\ref{thm:resultae}, for each $\eta>0$, if $\epsilon=\eta/2$, then $(1-\epsilon)f_1(n)-(1+\epsilon)f_2(m(n))\geq\frac{\eta^2}4f_1(n)$ for $n$ large enough. As $f_1(n)\to+\infty$ as $n\to\infty$, the theorem follows by Propositions~\ref{prop:limsupiinm} and~\ref{prop:liminm}.
\end{proof}

\begin{rmk} A situation analogous to that pointed out in Remark~\ref{rem:PP} holds also in connection with the above theorem, that is, the assumptions in Theorem~\ref{thm:resultinmeasure} are sufficient but not always necessary. Indeed, reasoning as in the proof of Proposition~\ref{prop:xPP}, one can prove that if $\mathcal P$ is log-balanced in measure and self-refining, then $L_n(x,\mathcal P,\mathcal P)/n\to 1$ in measure as $n\to\infty$.\end{rmk}

\section{Final remarks and open questions} \label{sec:conclusion} 

Theorem~\ref{thm:resultae} and~\ref{thm:resultinmeasure}
provide  asymptotic results for  Lochs' indexes for log-balanced sequences of partitions. %
While the assumption of log-bal\-anced\-ness could be regarded as relatively mild, it is crucial for our results. A natural question is whether we  can express sufficient conditions for log-balancedness, for instance in dynamical terms in the case of fibred systems. 

We have proved Lochs-type statements a.e.\ and in measure. Further studies concerning these Lochs-type results could be conducted, such as a multifractal analysis in the  spirit of \cite{BI:08bis}, or probabilistic estimates (central limit theorems) such as in \cite{Faivre:01,Wu:06}.

 One  motivation for the present work  comes from the experimental simulation of sources determined by sequences of partitions. (Here, by a source we mean the usual notion in information theory.)
   Lochs-type results come up naturally when we wonder how many of the initial digits in the binary expansion (which is computationally natural) of a random number are needed in order to deduce a prefix of a given length expressed in a different numeration system.
Let us first illustrate this with the  following question: how many binary digits of a random number do we need in order to generate the prefix of length $n$ of the characteristic Sturmian word (see Section \ref{subsec:sturm} for the definition)? We gave an answer to this question in Section~\ref{ssec:F&SB}.

Our work is also related to the probabilistic study of the behavior of data structures which are of fundamental importance in computer science. One such structure is the so-called \emph{trie}. A trie is a tree-like data-structure that
mimics the natural strategy to search a word
in a dictionary   by  comparing  words  via their prefixes.  Thus, the depth of a trie 
built from a set of words is related to ``coincidences'' between the words, that is,  to the
difficulty in distinguishing these words. For positive entropy, Cl\'ement, Flajolet and Vall\'ee \cite{DynSourceCFV01} show
that, when one builds a trie by picking $N$ random words produced by a reasonably well-made source with entropy $h>0$, its average depth has expectation asymptotically equal to $(1/h)\log N$. Furthermore, Cesaratto and Vall\'ee~\cite{MR3318040} have shown that its distribution is asymptotically Gaussian. In \cite{Aofa:20} the authors consider the average depth of the tries produced by random words produced by a source also as a mean to compare zero entropy sources.

We conjecture that there should be some relation between the weight function $f$, when it exists, and the average depth of the trie. %
We observe that the average depth of a trie is related to the inverse  of the weight function, for instance, in the case of positive entropy and for the Farey sequence. %
Note that $f^{-1}(\log N) = (1/h)\log N$, in the case of positive entropy $h$, as then we have the weight function $f(n)=hn$. For the case of  the Farey sequence of partitions, $f(n)=2 \log n$ and so $f^{-1}(\log N) = \sqrt{N}$. This coincides with the asymptotics for the expected value of the average depth of the Farey trie given in \cite{Aofa:20} up to a constant factor. %
However, this kind of connection does not hold for Stern-Brocot. %

\bigskip{
}
{\bf Acknowledgements.} We would like to thank Lo\"ick Lhote for discussions on the subject and Brigitte Vall\'ee for comments on earlier versions of this paper.

\newcommand{\etalchar}[1]{$^{#1}$}
\providecommand{\bysame}{\leavevmode\hbox to3em{\hrulefill}\thinspace}
\providecommand{\MR}{\relax\ifhmode\unskip\space\fi MR }
\providecommand{\MRhref}[2]{%
  \href{http://www.ams.org/mathscinet-getitem?mr=#1}{#2}
}
\providecommand{\href}[2]{#2}

\end{document}